\documentclass[a4paper,10pt,leqno]{amsart}
\title{Lehmer's Problem for arbitrary groups}
\author{L\"uck, W.}
        \address{Mathematicians Institut der Universit\"at Bonn\\
                Endenicher Allee 60\\
                53115 Bonn, Germany}
         \email{wolfgang.lueck@him.uni-bonn.de}
          \urladdr{http://www.him.uni-bonn.de/lueck}
         \date{March, 2020}

\usepackage{hyperref}
\usepackage{color}
\usepackage{pdfsync}
\usepackage{calc}
\usepackage{enumerate,amssymb}
\usepackage[arrow,curve,matrix,tips,2cell]{xy}
  \SelectTips{eu}{10} \UseTips
  \UseAllTwocells

\DeclareMathAlphabet\EuR{U}{eur}{m}{n}
\SetMathAlphabet\EuR{bold}{U}{eur}{b}{n}

\makeindex             

\theoremstyle{plain}
\newtheorem{theorem}{Theorem}[section]
\newtheorem{lemma}[theorem]{Lemma}

\newtheorem{conjecture}[theorem]{Conjecture}
\newtheorem{problem}[theorem]{Problem}

\theoremstyle{definition}

\newtheorem{definition}[theorem]{Definition}
\newtheorem{example}[theorem]{Example}

\newtheorem{remark}[theorem]{Remark}

{\catcode`@=11\global\let\c@equation=\c@theorem}

\newcommand{\comsquare}[8]                   
{\begin{CD}
#1 @>#2>> #3\\
@V{#4}VV @V{#5}VV\\
#6 @>#7>> #8
\end{CD}
}

\newcommand{\xycomsquare}[8]                   
{\xymatrix
{#1 \ar[r]^{#2} \ar[d]^{#4} &
#3 \ar[d]^{#5}  \\
#6\ar[r]^{#7} &
#8
}
}

\newcommand{\xycomsquareminus}[8]                      
{\xymatrix{#1 \ar[r]^-{#2} \ar[d]^-{#4} &
#3 \ar[d]^-{#5}  \\
#6\ar[r]^-{#7} &
#8
}
}

\newcommand{\cala}{\mathcal{A}}
\newcommand{\calb}{\mathcal{B}}
\newcommand{\cald}{\mathcal{D}}

\newcommand{\caln}{{\mathcal N}}
\newcommand{\calu}{{\mathcal U}}

\newcommand{\IC}{{\mathbb C}}

\newcommand{\IF}{{\mathbb F}}

\newcommand{\IQ}{{\mathbb Q}}
\newcommand{\IR}{{\mathbb R}}

\newcommand{\IZ}{{\mathbb Z}}

\newcommand{\colim}{\operatorname{colim}}

\newcommand{\id}{\operatorname{id}}

\newcommand{\im}{\operatorname{im}}

\newcommand{\tr}{\operatorname{tr}}

\newcommand{\vol}{\operatorname{vol}}

\newcommand{\higherlim}[3]{{\setbox1=\hbox{\rm lim}
        \setbox2=\hbox to \wd1{\leftarrowfill} \ht2=0pt \dp2=-1pt
        \mathop{\vtop{\baselineskip=5pt\box1\box2}}
        _{#1}}^{#2}#3}

\newcommand{\version}[1]                       
{\begin{center} last edited on #1\\
last compiled on \today\\
name of texfile: \jobname
\end{center}
}

\newcounter{commentcounter}

\newcommand{\squarematrix}[4]{\left( \begin{array}{cc} #1 & #2 \\ #3 &
#4
\end{array} \right)
}

\begin{document}

\typeout{----------------------------  lehmer_general.tex  ----------------------------}


\typeout{------------------------------------ Abstract ----------------------------------------}

\begin{abstract}
  We consider the problem whether for a group $G$ there exists a constant $\Lambda(G) >  1$ 
  such that for any $(r,s)$-matrix $A$ over the integral group ring $\IZ G$ the
  Fuglede-Kadison determinant of the $G$-equivariant bounded operator $L^2(G)^r \to
  L^2(G)^s$ given by right multiplication with $A$ is either one or greater or equal to
  $\Lambda(G)$. If $G$ is the infinite cyclic
  group and we consider only $r = s = 1$, this is precisely Lehmer's problem.
\end{abstract}

\maketitle


 \typeout{-------------------------------   Section 0: Introduction --------------------------------}

\setcounter{section}{-1}
\section{Introduction}

Lehmer's problem is the question whether the Mahler measure of a polynomial with
integer coefficients is either one or bounded from below by a fixed constant $\Lambda >1$. 
If one views the polynomial as an element in the integral group ring $\IZ[\IZ]$ ring of
$\IZ$, then its Mahler measure agrees with the Fuglede-Kadison determinant of the
$\IZ$-equivariant bounded operator $r_p^{(2)} \colon L^2(\IZ) \to L^2(\IZ)$ given by right
multiplication with $p$. This suggests to consider for any group  $G$ the following generalization.

Let $A$ be an $(r,s)$-matrix over the integral group ring $\IZ G$. We propose to study
the problem whether there is a constant $\Lambda(G) > 1$ such that the Fuglede-Kadison
determinant of $r^{(2)}_A \colon L^2(G)^r \to L^2(G)^s$ is either one or larger or equal
to $\Lambda(G)$. If we only allow $r = s = 1$, we denote such a constant by
$\Lambda_1(G)$.  If we consider only the case $r = s$ or the case $r = s =1$ and
additionally require that $r^{(2)}_A \colon L^2(G)^r \to L^2(G)^r$ is a weak isomorphism,
or,  equivalently, is injective, we denote such a constant by $\Lambda^w(G)$ or
$\Lambda_1^w(G)$. Lehmer's problem is equivalent to the question whether
$\Lambda_1(\IZ) > 1$ holds.

For obvious reasons we have $\Lambda(G) \le \Lambda^w(G)$,
$\Lambda_1(G) \le \Lambda_1^w(G)$, $\Lambda(G) \le \Lambda_1(G)$, and
$\Lambda^w(G) \le \Lambda_1^w(G)$.  Since for a group $G$ which contains $\IZ$ as a
subgroup we have $\Lambda_1^w(G) \le \Lambda_1^w(\IZ) = \Lambda_1(\IZ)$ by
Lemma~\ref{lem:elementary_properties_of_Lambda(G)}~%
\eqref{lem:elementary_properties_of_Lambda(G):subgroups} and
Theorem~\ref{the:Finitely_generated_free_abelian_groups}, we see that a counterexample to
Lehmer's problem would imply
$\Lambda(G) = \Lambda^w(G) = \Lambda_1(G) = \Lambda_1^w(G) =1$ for any group which
contains $\IZ$ as subgroup. Hence all the discussions in this paper are more or less void
if a counterexample to Lehmer's problem exists which is not known and fortunately not
expected to be true.

If there is no upper bound on the order of  finite subgroups of $G$, then $\Lambda^w_1(G) = 1$
by Remark~\ref{rem:bound_on_the_order_of_finite_subgroups}. Indeed, 
 there is a  finitely presented group $G$ with 
$\Lambda(G) =   \Lambda^w(G) = \Lambda_1(G) = \Lambda^w_1(G) = 1$, 
see Example~\ref{exa:lamplighther_group}. Therefore we will concentrate on torsionfree groups.

The most optimistic scenario would be  that for any torsionfree group $G$ all the constants
$\Lambda(G)$, $\Lambda^w(G)$, $\Lambda_1(G)$ and $\Lambda_1^w(G)$ are conjectured to be the
Mahler measure $M(L)$ of Lehmer's polynomial
$L(z) := z^{10} + z^9 - z^7 - z^6 -z^5 -z^4- z^3 + z +1$. But this is not the case in
general, there is a hyperbolic closed $3$-manifold, the Week's manifold W satisfying
$\Lambda^w(\pi_1(W)) < M(L)$, see Example~\ref{exa:Week's_manifold}.

We will not make any new contributions to Lehmer's problem in this article. However, we
think that it is interesting to put Lehmer's problem, which itself is already very
interesting and has many intriguing connections to number theory, topology and geometry,
in a more general context.  Moreover, we will give some evidence for the hope that
$\Lambda(G) > 1$ or  even $\Lambda(G) \ge \Lambda_1(\IZ)$ holds for some torsionfree
groups $G$. Namely, we will show in Theorem~\ref{the:Finitely_generated_free_abelian_groups} that 
$\Lambda^w(\IZ^d) = \Lambda_1(\IZ^d) = \Lambda^w_1(\IZ^d)$ holds for all
natural numbers $d \ge 1$ and that this value is actually independent of $d \ge 1$. We can also prove
$\Lambda(\IZ) = \Lambda(\IZ^d)$ for $d \ge 1$, but have not been able to relate
$\Lambda(\IZ)$ to $\Lambda^w(\IZ)$ expect for the obvious inequality $\Lambda(\IZ) \le \Lambda^w(\IZ)$.
In particular we do not know whether $\Lambda^w(\IZ) > 1 \implies \Lambda(\IZ) > 1$.  Conjecturally one may hope for
$\Lambda(\IZ) = \sqrt{\Lambda^w(\IZ)}$.

Moreover, we will
explain in Section~\ref{sec:General_Approximation_Results}
how to use  approximation techniques to potentially extend the class of groups for which
$\Lambda(G) > 1$ holds.

The paper has been  financially supported by the ERC Advanced Grant ``KL2MG-interactions''
 (no.  662400) of  the author granted by the European Research Council, and
 by the Deutsche Forschungsgemeinschaft (DFG, German Research Foundation)
under Germany's Excellence Strategy \--- GZ 2047/1, Projekt-ID 390685813, Cluster of Excellence 
``Hausdorff Center for Mathematics'' at Bonn. The author wants to thank the referees for their  detailed and useful reports.

\tableofcontents


 \typeout{------------------------   Section 1:  Lehmer's problem --------------------}

\section{Lehmer's problem}
\label{sec:Lehmers_Problem}

Let $p(z) \in \IC[\IZ] = \IC[z,z^{-1}]$ be a non-trivial element. 
Its \emph{Mahler measure} is defined by
\begin{eqnarray}
M(p) &:= &  \exp\left(\int_{S^1} \ln(|p(z)|) d\mu\right).
\label{Mahler_measure}
\end{eqnarray}
 By Jensen's equality we have
\begin{eqnarray}
\exp\left(\int_{S^1} \ln(|p(z)|) d\mu\right)
& = &
|c| \cdot \prod_{\substack{i=1,2, \ldots, r\\|a_i| > 1}} |a_i|,
\label{Jensen}
\end{eqnarray}
if we write $p(z)$ as a product
$p(z) = c \cdot z^k \cdot \prod_{i=1}^r (z - a_i)$
for an integer $r \ge 0$, non-zero complex numbers $c$, $a_1$, $\ldots$,  $a_r$ and an integer $k$.
This implies $M(p) \ge 1$ if $p$ has integer coefficients, i.e., belongs to $\IZ[\IZ] = \IZ[z,z^{-1}]$. 

The following problem goes back to a question of Lehmer~\cite{Lehmer(1933)}.

\begin{problem}[Lehmer's Problem]\label{pro:Lehmers_problem}
Does there exist  a constant $\Lambda > 1$ such that for all non-trivial
elements $p(z) \in \IZ[\IZ] = \IZ[z,z^{-1}]$ with $M(p) \not= 1$ we have
\[
M(p) \ge \Lambda.
\]
\end{problem}

\begin{remark}[Lehmer's polynomial]
\label{rem:Lehmers_polynomial}
There is even a candidate for which the minimal Mahler measure is attained,
namely, \emph{Lehmer's polynomial} 
\[
L(z) := z^{10} + z^9 - z^7 - z^6 -z^5 -z^4- z^3 + z +1.
\]
It is conceivable that for any non-trivial element $p \in \IZ[\IZ]$ with $M(p) > 1$ 
\[
M(p) \ge M(L) = 1.17628 \ldots \]
holds.

Actually, $L(z)$ is  $-z^5 \cdot \Delta(z)$, where  $\Delta(z)$ is the Alexander polynomial of the pretzel knot given by $(-2,3,7)$.
\end{remark}

For a survey on Lehmer's problem we refer for instance 
to~\cite{Boyd-Lind-Villegas-Deninger(1999),Boyd(1981speculations),Carrizosa(2009),Smyth(2008)}.


 \typeout{-------------   Section 2: The Mahler measure as Fuglede-Kadison determinant --------------------}

\section{The Mahler measure as Fuglede-Kadison determinant}
\label{sec:The_Mahler_measure_as_Fuglede-Kadison_determinant}

The following result is proved in~\cite[(3.23) on page~136]{Lueck(2002)}.
We will recall the Fuglede-Kadison determinant and its basic properties in the 
Appendix, see Section~\ref{sec:L2-invariants}.

\begin{theorem}[Mahler measure and Fuglede-Kadison determinants over $\IZ$] 
\label{lem:Mahler_measure_and_Fuglede-Kadison_determinant}
Consider an element $p = p(z) \in \IC[\IZ] = \IC[z,z^{-1}]$. It defines a bounded $\IZ$-equivariant operator
$r^{(2)}_p  \colon L^2(\IZ) \to L^2(\IZ)$ by multiplication with $p$. Suppose that $p$ is not zero.

Then the  Fuglede-Kadison determinant ${\det}_{\caln(\IZ)}(r_p^{(2)})$ of $r_p^{(2)}$ agrees with the Mahler measure, i.e.,
\[
{\det}_{\caln(\IZ)}(r_p^{(2)}) = M(p).
\]
\end{theorem}

Note that the identification of the Fuglede-Kadison determinant with the Mahler measure
holds also for non-trivial elements $p$ in $\IC[\IZ^d] = \IC[z_1^{\pm 1}, z_d^{\pm 1}, \ldots, z_d^{\pm 1}, ]$,
where $d$ is any natural number, see~\cite[Example~3.13 on page~128]{Lueck(2002)}.


 \typeout{-------------   Section 3:  Lehmer's problem for arbitrary groups  --------------------}

\section{Lehmer's problem for arbitrary groups}
\label{sec:Lehmer's_Question_for_arbitrary_groups}

Given a group $G$, we consider $L^2(G)$ as a Hilbert space with the obvious isometric
linear $G$-action from the left and write an element in $L^2(G)^r := \bigoplus_{i=1}^r L^2(G)$
as a row $(x_1,x_2, \ldots, x_r)$ for $x_i \in L^2(G)$, in other words as a $(1,r)$-matrix.
Given a $A$ in $M_{r,s}(\IZ G)$ or $M_{r,s}(\IC G)$, we obtain by right multiplication
with $A$ a bounded $G$-equivariant operator
\begin{eqnarray}
  r_A^{(2)} \colon L^2(G)^r & \to & L^2(G)^s, 
\quad  (x_i)_{i = 1,2, \ldots, r} \mapsto \left( \sum_{k = 1}^r x_k \cdot a_{k,j} \right)_{j = 1,2 \ldots, s}.
\label{r_A(2)}
\end{eqnarray}
Note that with these conventions we have $r_{AB}^{(2)} = r_B^{(2)} \circ r_A^{(2)}$ for  $A \in M_{r,s}(\IC G)$ 
and $B \in M_{s,t}(\IC G)$.

\begin{definition}[Lehmer's constant of a group]
\label{def:Lehmers_constants_of_a_group}
Define \emph{Lehmer's constant}  of a group $G$
\[
\Lambda(G) \in [1,\infty)
\]
to be the infimum of the set of Fuglede-Kadison determinants
\[
{\det}_{\caln(G)}\bigl(r_A^{(2)} \colon L^2(G)^r \to L^2(G)^s\bigr),
\]
where $A$ runs through all $(r,s)$-matrices $A \in M_{r,s}(\IZ G)$
for all $r,s \in \IZ$ with $r,s \ge 1$ for which ${\det}_{\caln(G)}(r_A^{(2)}) > 1$ holds.

If we only allow $(1,1)$-matrices $A$ with ${\det}_{\caln(G)}(r_A) > 1$, 
we denote the corresponding infimum by
\[
\Lambda_1(G) \in [1,\infty).
\]

If we only allow $(r,r)$-matrices $A$ for any natural number $r$ such that
$r_A^{(2)} \colon L^2(G)^r \to L^2(G)^r$ is a weak isomorphism, or, equivalently, is
injective, and ${\det}_{\caln(G)}(r_A) > 1$, we denote the corresponding infimum by
\[
\Lambda^w(G) \in [1,\infty).
\]

If we only allow $(1,1)$-matrices $A$ such that $r_A^{(2)} \colon L^2(G) \to L^2(G)$ is
weak isomorphism, or, equivalently, is injective, and ${\det}_{\caln(G)}(r_A) > 1$, then
we denote the corresponding infimum by
\[
\Lambda^w_1(G) \in [1,\infty).
\]
\end{definition}

Obviously we have 
\[
\begin{array}{lclcl}
\Lambda(G)  & \le &  \Lambda^w(G) & \le &  \Lambda_1^w(G);
\\
\Lambda(G)  & \le &  \Lambda_1(G) & \le &  \Lambda_1^w(G).
\end{array}
\]
A priori there is no obvious relation between $\Lambda^w(G)$ and $\Lambda_1(G)$.

\begin{problem}[Lehmer's problem for arbitrary groups]\label{pro:Lehmers_problem_for_arbitrary_groups}
  For which groups $G$ is $\Lambda(G) > 1$, $\Lambda_1(G) > 1$, $\Lambda^w(G) > 1$ or
  $\Lambda_1^w(G) > 1$ true?
\end{problem}

For amenable groups this problem is already considered in~\cite[Question~4.7]{Chung-Thom(2015)}.
See also~\cite{Dasbach-Lalin(2008),Dasbach-Lalin(2009)}.

\begin{remark}[$\Lambda_1(\IZ)$ and Lehmer's problem]\label{rem:Lambdaw_1(Zd)_and_Lehmers_problem}
    In view of Theorem~\ref{lem:Mahler_measure_and_Fuglede-Kadison_determinant} we see that 
  Lehmer's  Problem~\ref{pro:Lehmers_problem} is equivalent to the question whether
  $\Lambda_1(\IZ) > 1$.  In view of Remark~\ref{rem:Lehmers_polynomial} one would
  expect that $\Lambda_1(\IZ)$ is the Mahler measure $M(L)$ of Lehmer's polynomial. 
  We conclude $\Lambda_1(\IZ) = \Lambda^w_1(\IZ) = \Lambda^w(\IZ)$ 
  from Theorem~\ref{the:Finitely_generated_free_abelian_groups}. We do not know how
  $\Lambda(\IZ)$ and $ \Lambda^w(\IZ)$ are related except for the obvious inequality
  $\Lambda(\IZ) \le \Lambda^w(\IZ)$.
\end{remark}

\begin{remark}[Why matrices and why $\Lambda(G)$?]\label{rem:matrices}
  We are also interested besides $\Lambda_1(G)$ in the numbers $\Lambda^w(G)$, $\Lambda_1^w(G)$ and $\Lambda(G)$
  for the following reasons.  There is the notion of $L^2$-torsion, see for
  instance~\cite[Chapter~3]{Lueck(2002)}, which is essential defined in terms of the
  Fuglede-Kadison determinants of the differentials of the $L^2$-chain complex of the
  universal covering of a finite $CW$-complex or closed manifold. These differentials are
  given by $(r,s)$-matrices over $\IZ G$, where $r$ and $s$ can be any natural
  numbers. Therefore it is important to consider matrices and not only elements in $\IZ
  G$. Moreover, these differentials are not injective in general.

  Another reason to consider matrices is the possibility to consider restriction to a subgroup of finite index
since this passage turns a $(1,1)$-matrix into a matrix of the size $([G:H],[G:H])$.
  
  One advantage of $\Lambda(G)$ or $\Lambda_1(G)$ in comparison with $\Lambda^w(G)$ or $\Lambda^w_1(G)$ is the 
  better behavior under approximation, see Sections~\ref{sec:General_Approximation_Results}
  and~\ref{sec:Approximation_Results_over_Zd}. The problem is that for a square matrix $A$ over $G$
such that $r^{(2)}_A$ is a weak isomorphism, the operator $r_{p(A)}^{(2)}$ is not necessarily again  a weak isomorphism, if we have a 
not necessarily injective group homomorphism $p \colon G \to Q$ and $p(A)$ is the reduction of $A$ to a matrix over $Q$.
\end{remark}

\begin{remark}[Dobrowolski's estimate]\label{rem:Dobrowolskis_estimate}
  Dobrowolski~\cite{Dobrowolski(1991)} shows for a monic polynomial $p(z)$ with $p(0)
  \not= 0$ which is not a product of cyclotomic polynomials 
  \[
  M(p) \ge 1 + \frac{1}{a \exp(bk^k)},
  \]
  where $k$ is the number of non-zero coefficients of $p$ and $a$ and $b$ are given
  constants.  

 This triggers the question, whether for a given number $k$ and group $G$ there exists a
  constant $\Lambda_1(k,G) > 1$ such that for every element $x = \sum_{g \in G} n_g
  \cdot g$ in $\IZ G$ for which at most $k$ of the coefficients $n_g$ are not zero and
  ${\det}_{\caln(G)}(r_x^{(2)}) \not = 1$ holds, we have ${\det}_{\caln(G)}(r_x^{(2)}) \ge   \Lambda_1(k,G)$.  

 For $G = \IZ$ the existence of $\Lambda_1(k,G) > 1$  follows from Dobrowolski's result,
  and extends to $G =  \IZ^d$  by the  iterated limit appearing in
  Remark~\ref{rem:Approximating_Mahler_measures_for_polynomials_in_several_variables}.
\end{remark}


 \typeout{------------------------  Section 4:  The Determinant Conjecture --------------------------}

\section{The Determinant Conjecture}
\label{sec:The_Determinant_Conjecture}

Recall that the Mahler measure satisfies $M(p) \ge 1$ for any non-trivial polynomial $p$ with integer coefficients.
This is expected to be true for the Fuglede-Kadison determinant for all groups, namely, there is the 

\begin{conjecture}[Determinant Conjecture]
\label{con:Determinant_Conjecture}
Let $G$ be a group. Then for  any $A \in M_{r,s}(\IZ G)$ the Fuglede-Kadison determinant of the
morphism $r_A^{(2)}\colon L^2(G)^r \to L^2(G)^s$ given by right multiplication
with $A$ satisfies 
\[
{\det}_{\caln(G)}(r_A^{(2)}) \ge 1.
\]
\end{conjecture}

\begin{remark}[Status of the Determinant Conjecture]\label{rem:status_of_Determinant_Conjecture}
  The  following is known for the class $\cald$ of groups for which the Determinant
  Conjecture~\ref{con:Determinant_Conjecture} is true, see~\cite[Theorem~5]{Elek-Szabo(2005)},~%
\cite[Section~13.2 on pages~459~ff]{Lueck(2002)},~\cite[Theorem~1.21]{Schick(2001b)}.  
  \begin{enumerate}

  \item\label{rem:status_of_Determinant_Conjecture:amenable_quotient}
    Amenable quotient\\
    Let $H \subset G$ be a normal subgroup. Suppose that $H \in \cald$ and the
    quotient $G/H$ is amenable. Then $G \in \cald$;

  \item\label{rem:status_of_Determinant_Conjecture:direct_limit}
    Colimits\\
    If $G = \colim_{i \in I} G_i$ is the colimit of the directed system $\{G_i
    \mid i \in I\}$ of groups indexed by the directed set $I$ (with not
    necessarily injective structure maps) and each $G_i$ belongs to $\cald$,
    then $G$ belongs to $\cald$;

  \item\label{rem:status_of_Determinant_Conjecture:inverse_limit}
    Inverse limits\\
    If $G = \lim_{i \in I} G_i$ is the limit of the inverse system $\{G_i \mid i
    \in I\}$ of groups indexed by the directed set $I$ and each $G_i$ belongs to
    $\cald$, then $G$ belongs to $\cald$;

  \item\label{rem:status_of_Determinant_Conjecture:subgroups}
    Subgroups\\
    If $H$ is isomorphic to a subgroup of a group $G$ with $G \in \cald$, then
    $H \in \cald$;

  \item\label{rem:status_of_Determinant_Conjecture:quotient_with_finite_kernel}
    Quotients with finite kernel\\
    Let $1 \to K \to G \to Q \to 1$ be an exact sequence of groups. If $K$ is
    finite and $G$ belongs to $\cald$, then $Q$ belongs to $\cald$;

  \item\label{rem:status_of_Determinant_Conjecture:sofic_groups} Sofic groups
    belong to $\cald$.
  \end{enumerate}

  The class of sofic groups is very large.  It is closed under direct and free
  products, taking subgroups, taking inverse and direct limits over directed index sets, and 
  under extensions with amenable groups as quotients and a sofic group as
  kernel.  In particular it contains all residually amenable groups.  One
  expects that there exists non-sofic groups but no example is known.  More
  information about sofic groups can be found for instance
  in~\cite{Elek-Szabo(2006)} and~\cite{Pestov(2008)}.
\end{remark}

\begin{remark}[Invertible matrices and the Determinant Conjecture]
\label{rem:Invertible_matrices_Determinant_Conjecture}
Let $G$ be a group. Consider a matrix $A \in Gl_r(\IZ G)$. Then we get from
Theorem~\ref{the:main_properties_of_det}~\eqref{the:main_properties_of_det:composition}
\[
{\det}_{\caln(G)}(r_A^{(2)}) \cdot {\det}_{\caln(G)}(r_{A^{-1}}^{(2)})  = 1.
\]
If $G$ satisfies the Determinant Conjecture~\ref{con:Determinant_Conjecture}, we get
\begin{eqnarray}
{\det}_{\caln(G)}(r_A)  & = & 1 \quad \text{for}\; A \in Gl_r(\IZ G).
\label{det_(invertible)_is_1}
\end{eqnarray}

The argument appearing in the proof of~\cite[Theorem~6.7~(2)]{Lueck(2018)}
shows that the $K$-theoretic Farrell-Jones Conjecture for $\IZ G$ also implies~\eqref{det_(invertible)_is_1}.
\end{remark}


 \typeout{--   Section 5: Basic properties of Lehmer's constant for groups Lehmer's problem for arbitrary groups --}

\section{Basic properties of Lehmer's constant for arbitrary groups}
\label{sec:Basic_properties_of_Lehmers_constant_for_arbitrary_groups}

\begin{lemma}
\label{lem:elementary_properties_of_Lambda(G)}

\begin{enumerate}

\item\label{lem:elementary_properties_of_Lambda(G):subgroups}
If $H$ is a subgroup of $G$, then
\begin{eqnarray*}
\Lambda(G) & \le & \Lambda(H);
\\
\Lambda_1(G) & \le & \Lambda_1(H);
\\
\Lambda^w(G) & \le & \Lambda^w(H);
\\
\Lambda_1^w(G) & \le & \Lambda_1^w(H);
\end{eqnarray*}
\item\label{lem:elementary_properties_of_Lambda(G):finite_index}
If $H \subseteq G$ has finite index,  then  
\begin{eqnarray*}
\Lambda(H)^{[G:H]^{-1}} & \le & \Lambda(G);
\\
 \Lambda^w(H)^{[G:H]^{-1}} & \le & \Lambda^w(G);
\end{eqnarray*}

\item\label{lem:elementary_properties_of_Lambda(G):|G|_le_2}
We have
\[
\Lambda_1(\{1\}) =  \Lambda^w(\{1\}) =\Lambda^w_1(\{1\})  =  2,
\]
\[\Lambda(\{1\}) = \sqrt{2},
\]
and 
\[
\Lambda^w(\IZ/2) =  \Lambda^w_1 (\IZ/2)  =  \sqrt{3};
\]
\[
\Lambda_1(\IZ/2)  = \sqrt{2};
\]
\[
2^{1/4} \le \Lambda(\IZ/2) \le \sqrt{2};
\]
\item\label{lem:elementary_properties_of_Lambda(G):G_finite}
If $G$ is finite and $|G| \ge 3$, we get 
\[
\begin{array}{ccccccc}
2^{|G|^{-1}} &  \le &  \Lambda^w(G) & \le & \Lambda^w_1(G) &  \le  & (|G| -1)^{|G|^{-1}};
\\
  2^{(2|G|)^{-1}}  &  \le &  \Lambda(G) & \le &   \Lambda_1(G)  &\le  & (|G| -1)^{|G|^{-1}};
\end{array}
\]

\item\label{lem:elementary_properties_of_Lambda(G):finitely_generated_subgroups}
Let $G$ be a group. Then
\begin{eqnarray*}
\Lambda(G)  & = &\inf \{\Lambda(H) \mid H \subseteq G \;\text{finitely generated subgroup}\};
\\
\Lambda_1(G)  & = &\inf \{\Lambda_1(H) \mid H \subseteq G \;\text{finitely generated subgroup}\};
\\
\Lambda^w(G)  & = &\inf \{\Lambda^w(H) \mid H \subseteq G \;\text{finitely generated subgroup}\};
\\
\Lambda^w_1(G)  & = &\inf \{\Lambda^w_1(H) \mid H \subseteq G \;\text{finitely generated subgroup}\}.
\end{eqnarray*}

\end{enumerate}
\end{lemma}
\begin{proof}~\eqref{lem:elementary_properties_of_Lambda(G):subgroups}
Consider $A \in M_{r,s}(\IZ H)$. Let $i \colon H \to G$ be the inclusion.
By applying the ring homomorphism $\IZ H \to \IZ G$ induced by $i$ 
to the entries of $A$, we obtain a matrix $i_* A \in M_{r,s}(\IZ G)$. Then we get
\[
{\det}_{\caln(G)}(r^{(2)}_{i_*A}) = {\det}_{\caln(H)}(r_A^{(2)})
\]
from Theorem~\ref{the:main_properties_of_det}~\eqref{the:main_properties_of_det:induction}
and hence $\Lambda(H)  \le \Lambda(G)$ and $\Lambda_1(H)  \le \Lambda_1(G)$.

If $r = s$ and $r_A^{(2)}$ is injective, then $r^{(2)}_{i_*A}$ is injective because
of~\cite[Lemma~1.24~(3) on page~30]{Lueck(2002)}. This implies $\Lambda^w(G) \le
\Lambda^w(H)$ and $\Lambda^w_1(G) \le \Lambda^w_1(H)$.
\\[2mm]~\eqref{lem:elementary_properties_of_Lambda(G):finite_index} Consider a matrix 
$A \in M_{r,s}(\IZ G)$. We have introduced the bounded $G$-equi\-va\-riant operator 
$r_A^{(2)} \colon L^2(G)^r \to L^2(G)^s$ in~\eqref{r_A(2)}.  Let 
$i^* r_A^{(2)} \colon i^* L^2(G)^r \to i^* L^2(G)^s$ be the bounded 
$H$-equivariant operator obtained by restricting the
$G$-action to an $H$-action.  Since $[G:H]$ is finite, there is an $H$-equivariant isometric
isomorphism of Hilbert spaces from $L^2(H)^{[G:H]}$ to $i^* L^2(G)$. Hence for an
appropriate matrix $B \in M_{r \cdot [G:H],s \cdot [G:H]}(\IZ H)$ the bounded
$H$-equivariant operator $i^*r_A^{(2)} \colon i^*L^2(G)^r \to i^*L^2(G)^s$ can be
identified with $r_B^{(2)} \colon L^2(H)^{r \cdot [G:H]} \to L^2(H)^{s \cdot [G:H]}$. We
conclude from Theorem~\ref{the:main_properties_of_det}~\eqref{the:main_properties_of_det:restriction}
\[
{\det}_{\caln(H)}(r_{B}^{(2)}) = {\det}_{\caln(H)}(i^*r_{A}^{(2)}) = {\det}_{\caln(G)}(r_A^{(2)})^{[G:H]}.
\]
This implies $\Lambda(G)^{[G:H]} \ge \Lambda(H)$.

If $r = s$ and $r_A^{(2)}$ is injective, then $r_{i^*A}^{(2)}$ is injective. This implies
$\Lambda^w(G)^{[G:H]} \ge \Lambda^w(H)$.
\\[2mm]~\eqref{lem:elementary_properties_of_Lambda(G):|G|_le_2}
Consider $A \in M_{r,s}(\IZ) = M_{r,s}(\IZ[\{1\}])$.
Let $\lambda_1$, $\lambda_2$,
$\ldots$, $\lambda_r$ be the eigenvalues of $AA^*$ (listed with
multiplicity), which are different
from zero.  We get from Example~\ref{exa:det_for_finite_groups} 
\[
{\det}_{\caln(\{1\})}(r_A^{(2)}) = \sqrt{\prod_{i = 1}^r \lambda_i}.
\]
Let $p(t) = \det_{\IC}(t - AA^*)$ be the characteristic polynomial of
$AA^*$. It can be written as $p(t) = t^a \cdot q(t)$ for some polynomial
$q(t)$ with integer coefficients and $q(0) \not= 0$. One easily checks
\[
  |q(0)| = \prod_{i = 1}^r \lambda_i.
\]
Since $q$ has integer coefficients, we conclude $\det(r_A^{(2)}) = \sqrt{n}$ for some integer $n \ge 1$.
A direct calculation shows $\det(r_A^{(2)}) = \sqrt{2}$ for $A = \begin{pmatrix} 1 & 1 \\ 0 & 0 \end{pmatrix}$.
Hence we get
\[
\Lambda(\{1\}) = \sqrt{2}.
\]
Consider the square matrix $A \in M_{r,r}(\IZ) = M_{r,r}(\IZ[\{1\}])$ such that $r_A^{(2)}$ is a weak isomorphism, or, equivalently,
$A$ is invertible as a matrix over $\IC$. Then we conclude  from Example~\ref{exa:det_for_finite_groups}
\[{\det}_{\caln(\{1\})}(r_A^{(2)}) = |{\det}_{\IC}(A)| =  |{\det}_{\IZ}(A)| \in \{n \in \IZ \mid n \ge 1\}.
\]
This implies
\[
\Lambda_1(\{1\}) =  \Lambda^w(\{1\}) =\Lambda^w_1(\{1\})  =  2.
\]

Consider  $A \in M_{r,r}(\IZ[\IZ/2])$. It induces a $\IZ[\IZ/2]$-homomorphism
$r_A \colon \IZ[\IZ/2]^r \to \IZ[\IZ/2]^r$.
There exists an obvious short  exact sequence of $\IZ[\IZ/2]$-modules
$0 \to \IZ^- \to \IZ[\IZ/2] \to \IZ^+ \to 0$, where $\IZ$ is the underlying abelian group of $\IZ^{\pm}$ 
and the generator of $\IZ/2$ acts by $\pm \id$ on $\IZ^{\pm}$.  We obtain a commutative diagram of
endomorphisms of finitely generated free $\IZ$-modules
\[
\xymatrix{0 \ar[r]
&
\IZ[\IZ/2]^r \otimes_{\IZ[\IZ/2]} \IZ^-  \ar[r] \ar[d]^{r_A \otimes_{\IZ[\IZ/2]} \id_{\IZ^-}}
& 
\IZ[\IZ/2]^r  \ar[r] \ar[d]^{r_A}
& 
\IZ[\IZ/2]^r \otimes_{\IZ[\IZ/2]} \IZ^+  \ar[r] \ar[d]^{r_A \otimes_{\IZ[\IZ/2]} \id_{\IZ^+}}
& 
0
\\
0 \ar[r]
&
\IZ[\IZ/2]^r \otimes_{\IZ[\IZ/2]} \IZ^-  \ar[r]
& 
\IZ[\IZ/2]^r  \ar[r]
& 
\IZ[\IZ/2]^r \otimes_{\IZ[\IZ/2]} \IZ^+  \ar[r]
& 
0
}
\]
This implies
\[
{\det}_{\IZ}(r_A) = {\det}_{\IZ}(r_A \otimes_{\IZ[\IZ/2]} \id_{\IZ^-}) \cdot {\det}_{\IZ}(r_A \otimes_{\IZ[\IZ/2]} \id_{\IZ^+}).
\]
Since $\IZ^+ \otimes_{\IZ} \IF_2$ and
$\IZ^- \otimes_{\IZ} \IF_2$ are isomorphic as $\IF_2[\IZ/2]$-modules, we get
\[
{\det}_{\IF_2}\bigl(r_A \otimes_{\IZ[\IZ/2]} \id_{\IZ^-} \otimes_{\IZ} \id_{\IF_2}\bigr) =
{\det}_{\IF_2}\bigl(r_A \otimes_{\IZ[\IZ/2]} \id_{\IZ^+} \otimes_{\IZ} \id_{\IF_2}\bigr).
\]
Since the reduction to $\IF_2$ of
${\det}_{\IZ}(r_A \otimes_{\IZ[\IZ/2]} \id_{\IZ^\pm})$ is
${\det}_{\IF_2}\bigl(r_A \otimes_{\IZ[\IZ/2]} \id_{\IZ^{\pm}}\otimes_{\IZ} \id_{\IF_2}\bigr)$, we conclude
\[
\det(r_A \otimes_{\IZ[\IZ/2]} \id_{\IZ^-}) = \det(r_A \otimes_{\IZ[\IZ/2]} \id_{\IZ^+}) \mod 2.
\]
This implies that ${\det}_{\IZ}(r_A)$ is odd or divisible by four. In particular
$|{\det}_{\IZ}(r_A)|$ is different from $2$.  This implies that for 
any matrix $A \in M_{r,r}(\IZ[\IZ/2])$ for which $r_A \colon \IZ[\IZ/2] \to \IZ[\IZ/2]$ is injective, we have
$|{\det}_{\IZ}(r_A)| = 1$ or  $|{\det}_{\IZ}(r_A)| \ge 3$.

One easily checks that $|{\det}_{\IZ}(r_{t +2} \colon \IZ[\IZ/2] \to \IZ[\IZ/2])| = 3$. 
Since for any matrix $A \in M_{r,r}(\IZ[\IZ/2])$ with injective $r_A^{(2)} \colon  L^2(\IZ/2)^r \to  L^2(\IZ/2)^r $ we have
${\det}_{\caln(\IZ/2)}(r_{A}^{(2)}) = \sqrt{{|\det}_{\IZ}(r_{A})|}$ by Example~\ref{exa:det_for_finite_groups}, we conclude
$\Lambda^w(\IZ/2) = \Lambda^w_1(\IZ/2) = \sqrt{3}$.

Since ${\det}_{\caln(\IZ/2)}\bigl(r_{t+1} \colon L^2(\IZ/2) \to L^2(\IZ/2)\bigr) = \sqrt{2}$ holds by
Example~\ref{exa:det_for_(2,2)-matrices_over_the-trivial_group} and
Theorem~\ref{the:main_properties_of_det}~\eqref{the:main_properties_of_det:restriction}, we get
$\Lambda_1(\IZ/2) \le  \sqrt{2}$. We conclude $\sqrt{2} \le \Lambda_1(\IZ/2)$ 
from $\Lambda_1(\{1\}) = 2$ and
assertion~\eqref{lem:elementary_properties_of_Lambda(G):finite_index}.
This implies $\Lambda_1(\IZ/2) = \sqrt{2}$.

We conclude $2^{1/4} \le \Lambda(\IZ/2)$ 
from $\Lambda_1(\{1\}) = \sqrt{2}$ and
assertion~\eqref{lem:elementary_properties_of_Lambda(G):finite_index}.
\\[2mm]~\eqref{lem:elementary_properties_of_Lambda(G):G_finite}
We conclude from assertions~\eqref{lem:elementary_properties_of_Lambda(G):finite_index}
and~\eqref{lem:elementary_properties_of_Lambda(G):|G|_le_2} for the finite group $G$
\begin{eqnarray*}
  \Lambda(G) & \ge &  2^{(2|G|)^{-1}};
  \\
   \Lambda^w(G) & \ge &  2^{|G|^{-1}}.
\end{eqnarray*}
Consider the norm element $N_G := \sum_{g \in G} g$. Let $e \in G$ be the unit element.
Put $x = N_G -e \in \IZ G$. We have a canonical  $\IC G$-decomposition  $\IC G = \IC \oplus V$,
where $\IC$ is the trivial $G$-representation and $V$ is a direct sum of irreducible $G$-representations with
$V^G = 0$. Then $r_{N_G} \colon \IC G \to \IC G$ is the direct sum of $|G| \cdot \id \colon \IC \to \IC$
and $0 \colon V \to V$. Hence $r_x \colon \IC G \to \IC G$ is the direct sum of
$(|G|-1) \cdot \id_{\IC} \colon \IC \to \IC$ and of $-\id_V \colon V \to V$. This implies that 
$r_x \colon \IC G \to \IC G$ is a $\IC$-isomorphism and
\[
{\det}_{\IC}(r_x \colon \IC G \to \IC G) = |G|-1.
\]
We conclude  from by Example~\ref{exa:det_for_finite_groups}
\[
{\det}_{\caln(G)}(r_{x}^{(2)}) = (|G| -1)^{|G|^{-1}}.
\]
Since $|G| \ge 3$ and hence  $(|G| -1)^{|G|^{-1}}$ is different from $1$,we get
\[
\Lambda_1(G) \le \Lambda_1^w(G) \le (|G| -1)^{|G|^{-1}}.
\]
\\[2mm]~\eqref{lem:elementary_properties_of_Lambda(G):finitely_generated_subgroups}
We obtain 
\begin{eqnarray*}
\Lambda(G) & \le & \inf \{\Lambda(H) \mid H \subseteq G \;\text{finitely generated subgroup}\};
\\
\Lambda_1(G) & \le & \inf \{\Lambda_1(H) \mid H \subseteq G \;\text{finitely generated subgroup}\};
\\
\Lambda^w(G) & \le & \inf \{\Lambda(H) \mid H \subseteq G \;\text{finitely generated subgroup}\};
\\
\Lambda^w_1(G) & \le & \inf \{\Lambda^w_1(H) \mid H \subseteq G \;\text{finitely generated subgroup}\},
\end{eqnarray*}
from assertion~\eqref{lem:elementary_properties_of_Lambda(G):subgroups}.

Consider any matrix $A \in M_{r,s}(\IZ G)$. Let $H$ be the subgroup of $G$ which is generated
by the finite set consisting of those elements $g \in G$ for which for at least one entry in $A$ the coefficient of
$g$ is non-trivial. Then $H \subseteq G$ is finitely generated and $A = i_* B$ for some matrix
$B \in M_{r,s}(\IZ H)$ for the inclusion $i \colon H \to G$. We get
\[
{\det}_{\caln(G)}(r_{B}^{(2)}) = {\det}_{\caln(H)}(r_A^{(2)})
\]
from Theorem~\ref{the:main_properties_of_det}~\eqref{the:main_properties_of_det:induction}. This implies
\begin{eqnarray*}
\Lambda(G)  & \ge & \inf \{\Lambda(H) \mid H \subseteq G \;\text{finitely generated subgroup}\};
\\
\Lambda_1(G)  & \ge & \inf \{\Lambda_1(H) \mid H \subseteq G \;\text{finitely generated subgroup}\}.
\end{eqnarray*}
If $r = s$ and $r_A^{(2)}$ is injective, then also $r_B^{(2)}$ is injective.
Hence we get
\begin{eqnarray*}
\Lambda^w(G) & \ge & \inf \{\Lambda^w(H) \mid H \subseteq G \;\text{finitely generated subgroup}\};
\\
\Lambda^w_1(G)  &\ge & \inf \{\Lambda^w(H) \mid H \subseteq G \;\text{finitely generated subgroup}\}.
\end{eqnarray*}
This finishes the proof of Lemma~\ref{lem:elementary_properties_of_Lambda(G)}.
\end{proof}

\begin{example}[Finite cyclic group of odd order]
Let $n$ be an odd natural number. Then we get for the finite cyclic group $\IZ/n$  the equality
\begin{eqnarray}
\Lambda^w(\IZ/n)  = \Lambda_1^w(\IZ/n) = 2^{n^{-1}}.
\label{Lambda(Z/n)_for_odd_n_precise}
\end{eqnarray}
Namely, let $t \in \IZ/n$ be a generator. Consider the element $t + 1$. 
Then the $\IZ[\IZ/n]$-homomorphism $r_{t+ 1} \colon \IZ[\IZ/n] \to \IZ[\IZ/n]$ 
defines after forgetting the $\IZ/n$-action an $\IZ$-automorphism of $\IZ^n$
given by the matrix
\[
B[n] = 
\left(\begin{matrix}
1 & 0 & 0 & 0 & \cdots & 0 & 0 & 1
\\
1 & 1& 0 & 0 & \cdots & 0 & 0 & 0
\\
0 & 1 & 1 & 0 & \cdots & 0 & 0 & 0
\\
0& 0 & 1 & 1 & \cdots & 0 & 0 & 0
\\
0 & 0 & 0 & 1 & \cdots & 0 & 0 & 0
\\
\vdots & \vdots & \vdots & \vdots & \ddots & \vdots  & \vdots & \vdots
\\
0 & 0 & 0 &  0 & \cdots & 1 & 1 & 0
\\
0 & 0 & 0 & 0 & \cdots & 0 & 1 & 1
  \end{matrix}
  \right)
\]
Since $n$ is odd, we compute by developing after the first row 
$\det(B[n]) = 2$. Hence $r_{t+1}^{(2)}$ is injective and we get from
Example~\ref{exa:det_for_finite_groups}.
\[
{\det}_{\caln(\IZ/n)}(r_{t+1}^{(2)}) = 2^{n^{-1}}.
\]
This together with 
Lemma~\ref{lem:elementary_properties_of_Lambda(G)}~\eqref{lem:elementary_properties_of_Lambda(G):G_finite}
implies~\eqref{Lambda(Z/n)_for_odd_n_precise}.

Moreover, we get
\begin{eqnarray}
2^{(2n)^{-1}} \le \Lambda(\IZ/n)  \le \Lambda_1(\IZ/n)  \le  2^{n^{-1}}.
\label{Lambda(Z/n)_for_odd_n_estimate}
\end{eqnarray}
\end{example}

\begin{remark}[Computations for finite abelian groups]
  Lind~\cite[Definition~1.1]{Lind(2005)} has introduced a Lehmer
  constant for compact abelian groups.  If $G$ is a finite abelian
  group, then his constant agrees with  $\ln(\Lambda_1^w(G))$ for the number
  $\Lambda_1^w(G)$ introduced in
  Definition~\ref{def:Lehmers_constants_of_a_group}. Lind gives some precise values and
  some estimates for $\ln(\Lambda_1^w(G))$ for finite abelian groups which
  were considerably improved by Kaiblinger~\cite{Kaiblinger(2010)} for finite cyclic groups.

  Next we show that $\Lambda_1^w(G) = \Lambda^w(G)$ holds for finite abelian $G$. The
  classical determinant $\det_{\IC G}$ induces an isomorphism
  $K_1(\IC G) \xrightarrow{\cong} \IC G^{\times}$. The Fuglede-Kadison determinant
  $\det_{\caln(G)}$ induces a homomorphism $K_1(\IC G) \to \{r \in \IR \mid r > 0\}$. We
  have $\det_{\IZ G}(A) = \det_{\IC G}(A)$ for $A \in M_{r,r}(\IZ G)$.  For
    $A \in M_{r,r}(\IZ G)$ the map $r_A^{(2)}$ is a weak isomorphism if and only if it is
        an isomorphism, or, equivalently, $\det_{\IZ G}(A) = \det_{\IC G}(A)$ is a unit in
        $\IC G$. This implies for $A \in M_{r,r}(\IZ G)$ for which $r_A^{(2)}$ is a weak
        isomorphisms, that for $d := \det_{\IZ G}(A)$ the map $r_d^{(2)}$ is a weak
        isomorphism satisfying $\det_{\caln(G)}(r_A^{(2)}) = \det_{\caln(G)}(r_d^{(2)})$.
        Hence we get $\Lambda_1^w(G) \le \Lambda^w(G)$ and therefore $\Lambda_1^w(G) = \Lambda^w(G)$.

 In general we have $\Lambda_1(\IZ/n) \not= \Lambda_1^w(\IZ/n)$ and
 $\Lambda(\IZ/n) \not= \Lambda^w(\IZ/n)$, see 
Lemma~\ref{lem:elementary_properties_of_Lambda(G)}~\eqref{lem:elementary_properties_of_Lambda(G):|G|_le_2}.
\end{remark}

Computations for finite dihedral groups can be found in~\cite{Boerkoel-Pinner(2018)}.


\typeout{-------------   Section 6:  Torsionfree elementary amenable groups ----------------------------}

\section{Torsionfree elementary amenable groups}
\label{sec:Torsionfree_elementary_amenable_groups}

Throughout this  section let $G$ be an
amenable group for which $ \IQ G$ has no non-trivial zero-divisor. Examples for $G$ are
torsionfree elementary amenable groups,
see~\cite[Theorem~1.2]{Kropholler-Linnell-Moody(1988)},~\cite[Theorem~2.3]{Linnell(2006)}.
Then $\IQ G$ has a skewfield of fractions $S^{-1} \IQ G$ given by the Ore localization with
respect to the multiplicative closed subset $S$ of non-trivial elements in $\IQ G$,
see~\cite[ Example~8.16 on page~324]{Lueck(2002)}.











Next want to define a homomorphism
\begin{eqnarray}
\Delta \colon K_1(S^{-1}\IQ G) \to \IR^{>0}
\label{homomorphism_Delta}
\end{eqnarray}
as follows. Consider any natural number $r$ and a matrix $A \in GL_r(S^{-1}\IQ G)$.  
We can choose $a \in \IQ G$ with $a \not= 0$ such that $A[a] := (a \cdot I_r)
\cdot A$ belongs to $M_{r,r}(\IQ G)$, where $(a \cdot I_r)$ is the diagonal
$(r,r)$-matrix whose entries on the diagonal are all equal to $a$. Since 
$G$ satisfies the Determinant Conjecture~\ref{con:Determinant_Conjecture} by
Remark~\ref{rem:status_of_Determinant_Conjecture},  
the  Fuglede-Kadison determinants of
both $r_{A[a]} \colon L^2(G)^r \to L^2(G)^r$ and $r_{a} \colon L^2(G)^r \to L^2(GF)^r$ are  well-defined real numbers.
If $[A]$ denotes the class represented by $A$ in $K_1(S^{-1}\IQ G)$, we want to define
\[
\Delta([A]) := \frac{{\det}_{\caln(G)}\bigl(r_{A[a]}^{(2)}\bigr)}{{\det}_{\caln(G)}\bigl(r_{a\cdot I_r}^{(2)}\bigr)}.
\]
Note for the sequel that $r_{A[a]} \colon L^2(G)^r \to L^2(G)^r$ and
$r_{a \cdot I_r} \colon L^2(G)^r \to L^2(G)^r$ are weak isomorphisms by
Lemma~\ref{lem:det(C[G]}~\eqref{lem:det(C[G]:equivalent}.
The proof that that this is a well-defined homomorphism of abelian groups can be found in
in~\cite{Lueck(2018)} on the pages following (7.14), take $F = \IQ$ and $V$ to be
the trivial $1$-dimensional representation there.

There is a Dieudonne determinant for invertible matrices over a skewfield $K$
  which takes values in the
  abelianization of the group of units of the skewfield $K^{\times}/[K^{\times},K^{\times}]$ 
   and induces an isomorphism, see~\cite[Corollary~4.3 in page~133]{Silvester(1981)}
   \begin{eqnarray}
    {\det}_D \colon K_1(K) 
    & \xrightarrow{\cong} &
    K^{\times}/[K^{\times},K^{\times}].
   \label{Dieudonne_det_iso}
 \end{eqnarray}
   The inverse 
  \begin{eqnarray}
  \iota \colon  K^{\times}/[K^{\times},K^{\times}] 
  & \xrightarrow{\cong} &  
  K_1(K)
  \label{iota_K}
\end{eqnarray}
sends the class of a unit to the class of the corresponding $(1,1)$-matrix. 
In the sequel $K$ is chosen to be $S^{-1}\IQ G$.

The next result is a special case of~\cite[Lemma~7.23]{Lueck(2018)}, take $F = \IQ$ and $V$ to be the trivial $1$-dimensional representation.

\begin{lemma}\label{lem:det(C[G]}
Consider any matrix $A \in M_{r,r}(\IQ G)$. Then

\begin{enumerate}
\item\label{lem:det(C[G]:equivalent}
The following statements are equivalent:
\begin{enumerate}
\item\label{lem:det(C[G]:equivalent:r_A_inj}
$r_A \colon \IQ G^r \to \IQ G^r$ is injective;
\item\label{lem:det(C[G]:equivalent:r_A_S_inj}
$r_A \colon S^{-1}\IQ G^r \to S^{-1}\IQ G^r$ is injective;
\item\label{lem:det(C[G]:equivalent:r_A_S_bij}
$r_A \colon S^{-1}\IQ G^r \to S^{-1}\IQ G^r$ is bijective, or, equivalently $A$ becomes invertible over $S^{-1}\IQ G$;
\item\label{lem:det(C[G]:equivalent:r_A(2)_inj}
$r_A^{(2)} \colon L^2(G)^r \to L^2(G)^r$ is injective;
\item\label{lem:det(C[G]:equivalent:r_A(2)_weak}
$r_A^{(2)} \colon L^2(G)^r \to L^2(G)^r$ is a weak isomorphism;
\end{enumerate}
\item\label{lem:det(C[G]:dets}
If one of the equivalent conditions above is satisfied,  then $r_A^{(2)}$ is a weak
  isomorphism of determinant class and we get the equation
\[
{\det}_{\caln(G)}(r_A^{(2)})  =\Delta \circ \iota ({\det}_{D}(A)),
\]
where the homomorphisms $\Delta$ and $\iota$ have been defined in~\eqref{homomorphism_Delta}
and~\eqref{iota_K}. In particular ${\det}_{\caln(G)}(r_A^{(2)})$ agrees
with the quotient $\frac{{\det}_{\caln(G)}(r^{(2)}_x)}{{\det}_{\caln(G)}(r_y^{(2)})}$ 
for two appropriate elements $x,y \in \IQ G$ with $x,y \not= 0$.
\end{enumerate}
\end{lemma}

    On the first glance Lemma~\ref{lem:det(C[G]}~\eqref{lem:det(C[G]:dets} seems to be
    enough to show $\Lambda^w(G) = \Lambda^w_1(G)$ but this is not the case since we
    would need to replace
    $\frac{{\det}_{\caln(G)}(r^{(2)}_x)}{{\det}_{\caln(G)}(r_y^{(2)})}$ by
  ${\det}_{\caln(G)}(r^{(2)}_x)$. The following example illustrates why we do not know
  whether this is true in general.

  \begin{example}\label{exa:representatives_of_the_Dieudonne-determinant}
  If  $A = \begin{pmatrix} a & b \\ c &    d \end{pmatrix}$ is a $(2,2)$-matrix over a skewfield $K$, its
  Dieudonne determinant
  in $K^{\times}/[K^{\times},K^{\times}]$ is defined to be the class of $-cb$
  if $a = 0$ and to be the class of $ad -aca^{-1}b$ otherwise.  It can happen that the matrix 
  $A$ lives over $\IQ G$, but the obvious representative of the Dieudonne determinant does not. 
  The following example is due to Peter Linnell.  Let G be the metabelian group 
\[
\IZ \wr \IZ  = \langle x_i,y \mid  x_ix_j=x_jx_i, y^{-1}x_iy = x_{i+1} \:\text{for all}\;  i,j \in \IZ\rangle.
\]
Then we have $\IQ G \subset L^1(G) \subset \calu(G)$, and division ring of
quotients for $\IQ G$ is contained in $\calu(G)$, where $\calu(G)$ is the algebra of affiliated operators.
Consider the element $2-x_0 \in \IQ G$.  Then
$ (2-x_0)y(2-x_0)^{-1}$  is not contained in $\IQ G$ by the following observation.  This element is the same as
$y(1-x_1/2)(1-x_0/2)^{-1}$ and now we work inside $L^1(G)$, so we get
$y(1-x_1/2)(1+x_0/2+x_0^2/4+ \cdots)$. So the Dieudonne determinant of the matrix
$A = \begin{pmatrix} 2- x_0 & 1 \\ y &    0 \end{pmatrix}$ is represented by the element
$(2-x_0)y(2-x_0)^{-1}$  which is not contained in $\IQ G$ although all entries
of $A$ belong to $\IQ G$. 
\end{example}

\begin{remark}\label{rem:det_abelian_case}
If in the situation of Lemma~\ref{lem:det(C[G]} the group
$G$ happens to be abelian, then the Dieudonne determinant reduces to the standard determinant
${\det}_{\IQ G}$ for the commutative ring $\IQ G$ 
and it has the property that for a square matrix $A$ over $\IQ G$ its value ${\det}_{\IQ G}(A)$
is an element in $\IQ G$. Morerover, we can replace in Lemma~\ref{lem:det(C[G]}~\eqref{lem:det(C[G]:dets} 
the fraction $\frac{{\det}_{\caln(G)}(r^{(2)}_x)}{{\det}_{\caln(G)}(r_y^{(2)})}$ by
  ${\det}_{\caln(G)}(r^{(2)}_x)$ for some $x \in \IQ G$ with $x \not= 0$.
\end{remark}

We conclude from  Lemma~\ref{lem:det(C[G]}~\eqref{lem:det(C[G]:dets} and Remark~\ref{rem:det_abelian_case}

\begin{lemma}\label{lem:lambda_upper_w(Z_upper_d)_is_lambda_upper_w_1(Z_upper_d)}
We have $\Lambda^w(\IZ^d) = \Lambda^w_1(\IZ^d)$.
\end{lemma}

\begin{remark}\label{rem:representative_in_abelianized_units}
Example~\ref{exa:representatives_of_the_Dieudonne-determinant} does not rule out the possibility
that for every square matrix  $A$ over $\IQ G$, which is invertible  over $S^{-1}\IQ G$, there exists a non-trivial
element $u \in \IQ G$ such that the Dieudonne determinant of $A$ regarded as invertible matrix over the skewfield
$S^{-1}\IQ G$ in the abelian  group   $S^{-1}\IQ G^{\times}/[S^{-1}\IQ G^{\times},S^{-1}\IQ G^{\times}]$ is represented
by $u$. We neither have a proof for this claim nor a counterexample. This question is also interesting in
connection with the $L^2$-polytope homomorphism appearing in~\cite[Section~3.2]{Friedl-Lueck(2017universal)}.
Moreover, a positive answer implies~\cite[Theorem~5.14]{Kielak(2020)}.
\end{remark}


\typeout{-----------------------   Section 7: General Approximation Results---------------------------}

\section{General Approximation  Results}
\label{sec:General_Approximation_Results}

In this section we explain how approximation techniques may help in the future to extend the class of groups
for which one can give a positive answer to Lehmer's problem.  

\subsection{Approximation Conjecture for Fuglede-Kadison determinants}
\label{subsec:Approximation_Conjecture_for_Fuglede-Kadison_determinants}

We have  the following conjecture  which was formulated as  a question
in~\cite[Question~13.52 on  page~478]{Lueck(2002)}, see also~\cite[Section~15]{Lueck(2016_l2approx)}.

\begin{conjecture}[Approximation Conjecture for Fuglede-Kadison determinants]
\label{con:Approximation_conjecture_for_Fuglede-Kadison_determinants}
Let $G$ be a group together with an inverse system $\{G_i \mid i \in I\}$ of normal subgroups of $G$
directed by inclusion over the directed set $I$ such that $\bigcap_{i \in I} G_i = \{1\}$. Put $Q_i := G/G_i$.
Consider any matrix $A \in M_{r,s}(\IZ G)$. Denote by
$A_i \in M_{r,s}(\IZ Q_i)$ the reduction of $A$ to $\IZ Q_i $ coming from the projection $G \to Q_i$.
 
Then we  get for the Fuglede-Kadison determinants
\begin{eqnarray*}
\lefteqn{{\det}_{\caln(G)}\bigl(r_A^{(2)}\colon L^2(G)^r \to L^2(G)^s\bigr)}
& &
\\ & \hspace{14mm} =  &
\lim_{i \in I}\; {\det}_{\caln(Q_i)}\big(r_{A_i}^{(2)}\colon L^2(Q_i)^r \to L^2(Q_i)^s\bigr).
\end{eqnarray*}
\end{conjecture}

Unfortunately, the status of
Conjecture~\ref{con:Approximation_conjecture_for_Fuglede-Kadison_determinants} is very
poor, it does follow for virtually cyclic groups $G$ from the special case proved
in~\cite[Lemma~13.53 on page~478]{Lueck(2002)}, but we are not aware of a proof for a
group which is not virtually cyclic. Nevertheless there is hope that 
Conjecture~\ref{con:Approximation_conjecture_for_Fuglede-Kadison_determinants} is true
for torsionfree groups.

\begin{remark}[Integer coefficients are necessary]\label{rem:integer_coefficient_necessary}
  There are counterexamples to
  Conjecture~\ref{con:Approximation_conjecture_for_Fuglede-Kadison_determinants} if one
  replaces the coefficients in $\IZ$ by coefficients in $\IC$, see~\cite[Example~13.69 on
  page~481]{Lueck(2002)}.  This is in contrast to
  Theorem~\ref{the:Approximating_Mahler_measures_for_matrices_over_IC[IZd]_by_matrices_over_IC[IZ]}.
\end{remark}

Conjecture~\ref{con:Approximation_conjecture_for_Fuglede-Kadison_determinants} 
has the following interesting consequence.

\begin{theorem}[Consequence of the Approximation Theorem Conjecture for Fuglede-Kadison determinants]
\label{the:consequence_Approximation}
Let $G$ be a group together with an inverse system $\{G_i \mid i \in I\}$ of normal subgroups of $G$
directed by inclusion over the directed set $I$ such that $\bigcap_{i \in I} G_i = \{1\}$. Put $Q_i := G/G_i$.
Assume that each group $Q_i$ satisfies the Determinant Conjecturen~\ref{con:Determinant_Conjecture}.
Moreover, suppose that $G$ satisfies the Approximation Conjecture for Fuglede-Kadison 
determinants~\ref{con:Approximation_conjecture_for_Fuglede-Kadison_determinants}. Then
\begin{enumerate}
\item\label{the:consequence_Approximation:Lambda_andLambda_1}
We have
\begin{eqnarray*}
\Lambda(G) 
& \ge & 
\limsup_{i \in I} \Lambda(Q_i);
\\
\Lambda_1(G) 
& \ge &
\limsup_{i \in I} \Lambda_1(Q_i);
\end{eqnarray*}

\item\label{the:consequence_Approximation:Lambdaw}
  Suppose that for any element
  $A \in   M_{r,s}(\IZ G)$ there exists a constant $\beta(A) >0$ and an index $i_0(A) \in I$ such
  that the implication 
  $\dim_{\caln(Q_i)}\bigl(\ker(r_{A_i}^{(2)})\bigr) > 0   \implies \dim_{\caln(Q_i)}\bigl(\ker(r_{A_i}^{(2)})\bigr) \ge \beta(A)$ 
  holds for   all $i \in I$ with $i \ge i_0(A)$. (We will recall the notion of the von Neumann dimension $\dim_{\caln(G)}$ in  
Appendix~\ref{sec:L2-invariants}.)

Then we have
\[
\Lambda^w(G) \ge \limsup_{i \in I} \Lambda^w(Q_i);
\]
\item\label{the:consequence_Approximation:Lambda_1w} 
Suppose that for any element $A
  \in M_{1,1}(\IZ G)$ there exists a constant $\beta_1(A) >0$ and an index $i_0(A) \in I$
  such that the implication 
  $\dim_{\caln(Q_i)}\bigl(\ker(r_{A_i}^{(2)})\bigr) > 0   \implies \dim_{\caln(Q_i)}\bigl(\ker(r_{A_i}^{(2)})\bigr) \ge \beta_1(A)$ 
  holds for   all $i \in I$ with $i \ge i_0(A)$.

Then we have
\[
\Lambda^w_1(G) \ge \limsup_{i \in I} \Lambda_1^w(Q_i);
\]
\end{enumerate}
\end{theorem}

\begin{proof}~\eqref{the:consequence_Approximation:Lambda_andLambda_1}
This is obvious.
\\[2mm]~\eqref{the:consequence_Approximation:Lambdaw} and~\eqref{the:consequence_Approximation:Lambda_1w}
Consider a matrix $A \in M_{r,r}(\IZ G)$ such that $r_A^{(2)}$ is injective, or, equivalently,  
$\dim_{\caln(G)}\bigl(\ker(r_{A}^{(2)})\bigr)  = 0$.
We conclude  from~\cite[Theorem~13.19~(2) on page~461]{Lueck(2002)}
\[
0 = \dim_{\caln(G)}\bigl(\ker(r_{A}^{(2)})\bigr) = \lim_{i \in I} \dim_{\caln(G)}\bigl(\ker(r_{A_i}^{(2)})\bigr).
\]
Hence there exists $i_0$ such that $\dim_{\caln(G)}\bigl(\ker(r_{A_i}^{(2)})\bigr) = 0$ holds for $i \ge i_0$
and hence that $r_{A_i}^{(2)}$ is injective for $i \ge i_0$. 
\end{proof}

\begin{remark}[Atiyah Conjecture]\label{rem:Atiyahs_Conjecture}
  A version of the \emph{Atiyah Conjecture} says for a group $G$ for which there exists a
  natural number $D$ such that the order of any finite subgroup of $G$ divides $D$ that
  for any element $A \in M_{r,s}(\IZ G)$ we get
  \[
  D \cdot \dim_{\caln(G)}\bigl(\ker(r_A^{(2)} \colon L^2(G)^r \to L^2(G)^s)\bigr) \in \IZ.
  \]
  If $G$ happens to be torsionfree, we can choose $D = 1$ and get the implication
  $\dim_{\caln(G)}\bigl(\ker(r_A^{(2)} \colon L^2(G)^r \to L^2(G)^s)\bigr) \in \IZ$.

  Suppose that there exists a natural number $D$ such that for every $i \in I$ and every finite
  subgroup $H \subseteq Q_i$ the order $|H|$ divides $D$. Then the implication
  $\dim_{\caln(Q_i)}\bigl(\ker(r_{A_i}^{(2)})\bigr) > 0 \implies
  \dim_{\caln(Q_i)}\bigl(\ker(r_{A_i}^{(2)})\bigr) \ge \beta(A)$ appearing in
  assertions~\eqref{the:consequence_Approximation:Lambdaw}
   and~\eqref{the:consequence_Approximation:Lambda_1w} of
  Theorem~\ref{the:consequence_Approximation} is automatically satisfied if each group $Q_i$
  satisfies the Atiyah Conjecture above, just take $\beta(A) := \frac{1}{D}$. 

  A survey on the Atiyah Conjecture and the groups for which it is known to be true can be
  found in~\cite[Chapter~10]{Lueck(2002)}. We mention that the Atiyah Conjecture holds for
  $G$ if $G$ is elementary amenable and there is a upper bound on the orders of the finite subgroups of  $G$.
\end{remark}

We conclude from Theorem~\ref{the:consequence_Approximation}
and Remark~\ref{rem:Atiyahs_Conjecture}

\begin{theorem}[Residually torsionfree elementary amenable groups]%
\label{the:Residually_torsionfree_elementary_amenable_groups}
  Let $G$ be a residually torsionfree elementary amenable group in the sense that we can
  find a sequence of in $G$ normal subgroups
  $G = G_0 \supseteq G_1 \supseteq G_2 \supseteq \cdots$ such that $G/G_n$ is torsionfree
  elementary amenable for $n = 0,1,2, \ldots$ and $\bigcap_{n \ge 0} G_n = \{1\}$.
  Suppose that $G$ satisfies the Approximation
  Conjecture~\ref{con:Approximation_conjecture_for_Fuglede-Kadison_determinants} for
  Fuglede-Kadison determinants.

  Then
  \begin{eqnarray*}
    \Lambda^w(G)  & \ge & \limsup_{i \in I} \Lambda^w(Q_i);
    \\
    \Lambda^w_1(G)  & \ge & \limsup_{i \in I} \Lambda_1^w(Q_i).
  \end{eqnarray*}
\end{theorem}

\begin{example}[Examples of (virtually) residually torsionfree nilpotent]%
\label{exa::(virtually)_residually_torsionfree_nilpotent_groups}
  Free groups are examples of residually
  torsionfree nilpotent groups, see~\cite[\S~2]{Magnus(1935)}.

  Let $M$ be a compact orientable irreducible $3$-manifold whose boundary is empty or is a
  disjoint union of incompressible tori and whose fundamental group $\pi$ is infinite and not solvable. Then its
  fundamental group $\pi$ is 
  virtually residually torsionfree nilpotent, see  for instance~\cite[page~84]{Aschenbrenner-Friedl-Wilton(2015)}.

  This implies that $\Lambda(\pi) > 1$ if $\pi$  satisfies the Approximation
  Conjecture~\ref{con:Approximation_conjecture_for_Fuglede-Kadison_determinants}
  and $\Lambda^w(H) = \Lambda(\IZ)$ holds for any torsionfree nilpotent group. Note that we cannot conclude
   $\Lambda(\pi) = \Lambda(\IZ)$ since we only know that $\pi$ is 
virtually residually torsionfree nilpotent and not that $\pi$ is residually torsionfree nilpotent.

\end{example}

\begin{remark} In order to apply
  Theorem~\ref{the:Residually_torsionfree_elementary_amenable_groups} one needs to know
  Conjecture~\ref{con:Approximation_conjecture_for_Fuglede-Kadison_determinants} which we
  have already discussed above and also have some information for $\Lambda^w(H)$, for
  nilpotent groups. Not much is known for these groups. Not even the three-dimensional
  Heisenberg group is fully understood.  See for  instance~\cite[Section~5]{Deninger(2011)}.
\end{remark}


\subsection{Sub-Approximation Theorem}
\label{subsec:Approximation_Theorem}

At least we can  prove an inequality in a situation which is more general than the case of a normal chain
considered in Subsection~\ref{subsec:Approximation_Conjecture_for_Fuglede-Kadison_determinants}.

Given a matrix $A \in M_{r,s}(\IC G)$, we will in the sequel denote by $A^*$ the element in $M_{s,r}(\IC G)$,
whose $(i,j)$-th entry is $a_{j,i}^*$,
where for an element $x = \sum_{g \in G} \lambda_g \cdot g \in \IC G$ we denote by
$x^*$ the element $\sum_{g \in G} \overline{\lambda_g} \cdot g^{-1} \in \IC G$.
With this convention the adjoint $(r_A^{(2)})^*$ of the bounded operator
$r_A^{(2)} \colon L^2(G)^r \to L^2(G)^s$ is $r_{A^*}^{(2)} \colon L^2(G)^s \to L^2(G)^r$.

\begin{theorem}[Sub-Approximation Theorem]\label{the:subApproximation}
Let $G$ be a group.   Suppose that $I$ is a directed set and we have a collection of groups $\{Q_i \mid i \in I\}$
together with group homomorphisms $q_i \colon G \to Q_i$ for each $i \in I$.
 Given $A \in M_{r,s}(\IZ G)$, denote by
$A_i \in M_{r,s}(\IZ Q_i)$ the reduction of $A$ to $\IZ Q_i $ coming from the projection $G \to Q_i$.
Suppose:
\begin{itemize}
\item For any finite subset $F \subseteq G$ there exists an index $i_0(F) \in I$ such that for all $i \ge i_0(F)$ and $f \in F$
the implication $q_i(f) \not= e \implies f \not= e$ holds,
where $e$ denotes the unit element in $G$ and  $Q_i$; (This is automatically satisfied if
there is an inverse system $\{G_i \mid i \in I\}$ of normal subgroups of $G$
directed by inclusion over the directed set $I$ such that $\bigcap_{i \in I} G_i = \{1\}$, $Q_i = G/G_i$ and
$q_i \colon G \to Q_i$ is the projection.)

\item Each group $Q_i$ satisfies the Determinant Conjecture~\ref{con:Determinant_Conjecture}.

\end{itemize}

Then for any element $A \in M_{r,s}(\IZ G)$ we have
\[
{\det}_{\caln(G)}(r_A^{(2)}) \ge \limsup_{i \in I} {\det}_{\caln(Q_i)}(r_{A_i}^{(2)}).
\]

\end{theorem}
\begin{proof}
We conclude from Theorem~\ref{the:main_properties_of_det}~\eqref{the:main_properties_of_det:det(f)_is_det(fast)}
\begin{eqnarray*}
{\det}_{\caln(G)}(r_A)
& = & 
\sqrt{{\det}_{\caln(G)}(r_{AA^*}^{(2)})};
\\
{\det}_{\caln(Q_i)}(r_{A_i}^{(2)})
& = & 
\sqrt{{\det}_{\caln(Q_i)}(r_{A_iA_i^*}^{(2)})}.
\end{eqnarray*}

Hence we can assume without loss of generality that $r = s$ and that $r_A^{(2)} \colon L^2(G)^r \to L^2(G)^r$ 
and $r_{A_i}^{(2)} \colon L^2(Q_i)^r \to L^2(Q_i)^r$  are positive operators, otherwise replace $A$ by $AA^*$.

We want to apply~\cite[Theorem~13.19 on page~461]{Lueck(2002)}  in the case, where 
$G_i$ in the notation of~\cite[Theorem~13.19 on page~461]{Lueck(2002)} is $Q_i$,
$A_i \in M_{r,r}(\IZ Q_i)$ is the matrix above and $\tr_i = \tr_{\caln(G_i)}$. Then the claim does not follow directly
from the assertions in~\cite[Theorem~13.19 on page~461]{Lueck(2002)} but from the inequality 
\[
\ln\left({\det}_{\caln(G)}(r_A^{(2)})\right) \ge \limsup_{i \in I} \ln\left({\det}_{\caln(Q_i)}(r_{A_i}^{(2)})\right)
\]
appearing at the very end of the proof  of~\cite[Theorem~13.19 on page~465]{Lueck(2002)}.
It remains to check all the assumptions appearing in~\cite[Theorem~13.19 on page~461]{Lueck(2002)}.

Choose a real number $K$ satisfying
\[
K \ge  \sqrt{(2r-1) \cdot r}\cdot \max\left\{ ||A_{j,k}||_1 \mid 1 \le j \le r, 1 \le k \le s\right\}.
\]
Then we also have 
\[
K \ge  \sqrt{(2r-1) \cdot r} \cdot \max\left\{ ||(A_i)_{j,k}||_1 \mid 1 \le j \le r, 1 \le k \le s\right\}.
\]
This implies the  inequalities $||r_A^{(2)}|| \le K$ and $||r_{A_i}^{(2)}|| \le K$
for the operator norms of $r_A^{(2)}$ and $r_{A_i}^{(2)}$. 

Consider a polynomial $p$ with real coefficients. 
Let $F$ be the finite set of elements $g$ in $G$ for which there exists a natural number  $j$ with
$1 \le j \le r$ such that the coefficient of $g$ of the 
$(j,j)$-th entry of $p(A)$ is different from zero. Choose $i_0 \in I$ such for all
$i \ge i_0$ and $f \in F$ the implication $q_i(f) \not= e \implies f \not= e$ holds. This implies
\[
\tr_{\caln(G)}(p(A)) = \tr_{\caln(Q_i)}(p(A_i)) \quad \text{for}\; i \ge i_0.
\]
In particular we get
\[
\tr_{\caln(G)}(p(A)) = \lim_{i \in I} \tr_{\caln(Q_i)}(A_i).
\]
This finishes the proof of Theorem~\ref{the:subApproximation}.
\end{proof}

\begin{remark}[Sub-Approximation and Lehmer's
  problem]\label{rem:sub_approximation_and_Lehmers_problem}
  Theorem~\ref{the:subApproximation} looks more promising than
  Theorem~\ref{the:consequence_Approximation} if one is interested in the
  conclusion~\eqref{the:consequence_Approximation:Lambda_andLambda_1} of
  Theorem~\ref{the:consequence_Approximation} only. 
  Theorem~\ref{the:consequence_Approximation} applies to more general systems of group
  homomorphisms $g \to Q_i$ than in Theorem~\ref{the:consequence_Approximation} and it
  does not need in contrast to Theorem~\ref{the:consequence_Approximation} the condition
  that $G$ satisfies the Approximation Conjecture for Fuglede-Kadison
  determinants~\ref{con:Approximation_conjecture_for_Fuglede-Kadison_determinants}, but
  only requires only a milder version to imply the same conclusion as in
  Theorem~\ref{the:consequence_Approximation}~\eqref{the:consequence_Approximation:Lambda_andLambda_1}.
  
  Namely, we have additionally to assume that ${\det}_{\caln(G)}(r_A^{(2)}) > 1$ implies
  the existence of an index $i_0$ such that ${\det}_{\caln(Q_i)}(r_{A_i}^{(2)}) > 1$ holds
  for $i \ge i_0$, because then ${\det}_{\caln(Q_i)}(r_{A_i}^{(2)}) > \Lambda(Q_i)$ holds
  for $i \in I$ with $i \ge i_0$ and hence the inequality 
  $\limsup_{i \in I} {\det}_{\caln(Q_i)}(r_{A_i}^{(2)}) \ge \limsup_{i \in I}\Lambda(Q_i)$
  is true.
  
  Or we have additionally to assume that
  $\limsup_{i \in I} {\det}_{\caln(Q_i)}(r_{A_i}^{(2)}) = 1$ implies
  ${\det}_{\caln(G)}(r_A^{(2)}) = 1$ and that there exists a number $\Lambda > 1$ and an
  index $i_0$ such that $\Lambda(Q_i) \ge \Lambda$ holds for $i \ge i_0$, because then
  either $\limsup_{i \in I} {\det}_{\caln(Q_i)}(r_{A_i}^{(2)}) = 1$ or the inequality
  $\limsup_{i \in I} {\det}_{\caln(Q_i)}(r_{A_i}^{(2)}) \ge \Lambda$ holds.
\end{remark}


\typeout{-----------------------   Section 8: Approximation Results over $\IZ^d$ ---------------------------}

\section{Approximation  Results over $\IZ^d$}
\label{sec:Approximation_Results_over_Zd}

For $G = \IZ^d$ we get  much better approximation results.
Essentially we will generalize the approximation results
of Boyd and Lawton to arbitrary matrices over $\IZ[\IZ^d]$. This will be important for the
proof of Theorem~\ref{the:Finitely_generated_free_abelian_groups}.

\begin{remark}[Approximating Mahler measures for polynomials in several variables by polynomials in one variable]%
\label{rem:Approximating_Mahler_measures_for_polynomials_in_several_variables}
  There is a case, where the inequality in Theorem~\ref{the:subApproximation} becomes an
  equality with $\limsup$ replaced by $\lim$. Namely, let $p(z_1,z_2)$ be a non-trivial
  polynomial with complex coefficients in two variables $z_1$ and $z_2$.  For a natural
  number $k$ let $p(z,z^k) $ be the polynomial with complex coefficients in one variable
  $z$ obtained from $p(z_1,z_2)$ by replacing $z_1 = z$ and $z_2 = z^k$ in $p(z_1,z_2)$.
  This corresponds to the homomorphism $q_k \colon \IZ^2 \to \IZ, \; (n_1,n_2) \mapsto n_1  + k \cdot n_2$.
  If $k$ is large enough, then $p(z,z^k)$ is again non-trivial. We have the formula
  \[{\det}_{\caln(\IZ^2)}\bigl(r_{p(z_1,z_2)}^{(2)} \colon L^2(\IZ^2) \to L^2(\IZ^2)\bigr)
  = \lim_{k \to \infty} {\det}_{\caln(\IZ)}\bigl(r_{p(z,z^k)}^{(2)} \colon L^2(\IZ) \to
  L^2(\IZ)\bigr).
  \]
  Its proof can be found in~\cite[Appendix~3]{Boyd(1981speculations)}.   The corresponding
  formula for a non-trivial polynomial $p(z_1,z_2, \ldots, z_d)$ with complex coefficients
  in $d$-variables $z_1$, $z_2$, $\ldots$, $z_d$
  \begin{multline*} {\det}_{\caln(\IZ^d)}\bigl(r_{p(z_1, z_2, \ldots , z_d)}^{(2)} \colon
    L^2(\IZ^d) \to L^2(\IZ)^d\bigr)
    \\
    = \lim_{k_2 \to \infty} \; \lim_{k_3 \to \infty} \; \ldots \; \lim_{k_d \to \infty} \;
    {\det}_{\caln(\IZ)}\bigl(r_{p(z, z^{k_2}, \ldots , z_r^{k_d})}^{(2)} \colon L^2(\IZ)
    \to L^2(\IZ)\bigr)
  \end{multline*}
  is  proved in~\cite[Appendix~4]{Boyd(1981speculations)}  and~\cite[Theorem~2]{Lawton(1983)}.
  Note that for given $p$ and natural numbers $k_2$, $k_3$, \ldots, $k_d$,
  we can find natural   numbers $N_2$, $N_3(k_2)$, \ldots, $N_d(k_2,k_3, \ldots, k_{d-1})$
  such that we have $p(z, z^{k_2}, \ldots , z_r^{k_d}) \not= 0$,
  provided that $k_2 \ge N_2$, $k_3 \ge N_3(k_2)$, \ldots, $k_d \ge N_d(k_2, k_3, \ldots ,k_d)$ hold.
\end{remark}

Consider a natural number $d \ge 2$. For natural numbers $k_2, k_3, \ldots, k_d$ define a
group homomorphism
\[
q(k_2, k_3, \ldots k_d) \colon \IZ^d \to \IZ, \quad (a_1, a_2, \ldots, a_d) \mapsto
a_1 + \sum_{i = 2}^d k_i \cdot a_i.
\]
It induces a ring homomorphism $\widehat{q}(k_2, k_3, \ldots k_d) \colon \IC[\IZ^d] \to \IC[\IZ]$.  
Given a matrix $A$, let 
\[
A[k_2,\ldots, k_d] \in M_{r,s}(\IC[\IZ])
\] 
be the   matrix obtained from $A$ by applying the ring homomorphism $\widehat{q}(k_2, k_3, \ldots  k_d)$ 
   to each entry. If $p = p(z_1, z_2, \ldots, z_d)$ is polynomial and we regard it
  as element in $M_{1,1}(\IC[\IZ^d])$, then $p[k_2, \ldots, k_d]$ is the $(1,1)$-matrix
    over $\IC[\IZ]$ given by the polynomial $p(z,z^{k_2}, \ldots, z^{k_d})$. 
Note that for any element $p \in \IC[\IZ^d] = \IC[z^{\pm 1}, \ldots, z^{\pm 1}_d]$ 
we can find a monomial $z_1^{n_1} \cdot z_2^{n_2} \cdot \cdots \cdot z_d^{n_d}$
such that $p_0 := z_1^{n_1} \cdot z_2^{n_2} \cdot \cdots \cdot z_d^{n_d} \cdot p$ is a polynomial in variables
$z_1$, $\ldots$, $z_d$ and we have ${\det}_{\caln(\IZ^d)}(r_{p}^{(2)}) = {\det}_{\caln(\IZ^d)}(r_{p_0}^{(2)})$. Hence we
    conclude from the iterated limit appearing in
    Remark~\ref{rem:Approximating_Mahler_measures_for_polynomials_in_several_variables}
    that for a $(1,1)$-matrix $A$ over $\IC[\IZ^d]$ we have
    \begin{multline}
   {\det}_{\caln(\IZ^d)}\bigl(r_A^{(2)} \colon L^2(\IZ^d) \to L^2(\IZ^d)\bigr) 
    \\ =
    \lim_{k_2 \to \infty} \; \lim_{k_3 \to \infty} \; \ldots \; \lim_{k_d \to \infty} \;
    {\det}_{\caln(\IZ)}\bigl(r_{A[k_2, \ldots, k_d]}^{(2)} \colon L^2(\IZ) \to
    L^2(\IZ)\bigr).
    \label{approximation_rewritten_for_(1,1)-matrices}
  \end{multline}
We want to extend this to matrices of arbitrary finite size. For this purpose the following formula is useful, which reduces
the computation of the Fuglede-Kadison determinant ${\det}_{\caln(\IZ^d)}(r_A^{(2)})$ for $A \in M_{r,s}(\IC[\IZ^d])$ to 
the special case, where $r = s$ and $r_A^{(2)}$ is injective, and finally to the case $r = s = 1$.

\begin{lemma}\label{lem:reducting_to_weak_autos}
Consider a matrix $A \in M_{r,s}(\IC[\IZ^d])$. Then there exists an integer $q \ge 0$ and a matrix $B \in M_{q,r}(\IC[\IZ^d])$
such that the sequence $0 \to L^2(\IZ^d)^q \xrightarrow{r_B^{(2)}} L^2(\IZ^d)^r \xrightarrow{r_A^{(2)}} L^2(\IZ^d)^s$
is weakly exact, i.e., $r_B^{(2)}$ is injective and the closure of the image of $r_B^{(2)}$ is the kernel of $r_A^{(2)}$.
For any such choice we get for the matrices $D_1 = B^*B + AA^*$ in $M_{r,r}(\IC[\IZ^d])$ and $D_2 = BB^*$ in $M_{q,q}(\IC[\IZ^d])$:

\begin{enumerate}

\item\label{lem:reducting_to_weak_autos:weak}
The operators $r_{D_1}^{(2)} \colon L^2(\IZ^d)^r \to L^2(\IZ^d)^r$ and
$r_{D_2}^{(2)} \colon L^2(\IZ^d)^q \to L^2(\IZ^d)^q$ are injective;

\item\label{lem:reducting_to_weak_autos:formula}
We have
\[{\det}_{\caln(\IZ^d)}(r_A^{(2)}) =
\sqrt{\frac{{\det}_{\caln(\IZ^d)}(r_{D_1}^{(2)})}{{\det}_{\caln(\IZ^d)}(r_{D_2}^{(2)} )}};
\]

\item\label{lem:reducting_to_weak_autos:reducing_with_det_over_C[Zd]}
The elements ${\det}_{\IC[\IZ^d]}(D_1)$ and ${\det}_{\IC[\IZ^d]}(D_2)$ in $\IC[\IZ^d]$ are different 
from zero and we get
\begin{eqnarray*}
{\det}_{\caln(\IZ^d)}(r_{D_1}^{(2)}) 
& = & 
{\det}_{\caln(\IZ^d)}\bigl(r_{{\det}_{\IC[\IZ^d]}(D_1)}^{(2)} \colon L^2(\IZ^d) \to L^2(\IZ^d)\bigr);
\\
{\det}_{\caln(\IZ^d)}(r_{D_2}^{(2)}) 
& = & 
{\det}_{\caln(\IZ^d)}\bigl(r_{{\det}_{\IC[\IZ^d]}(D_2)}^{(2)} \colon L^2(\IZ^d) \to L^2(\IZ^d)\bigr).
\end{eqnarray*}
\end{enumerate}
\end{lemma}
\begin{proof}~\eqref{lem:reducting_to_weak_autos:weak} Let $\IC[\IZ^d]_{(0)}$ be the
  quotient field of $\IC[\IZ^d]$. Let $q$ be the $\IC[\IZ^d]_{(0)}$-dimension of the
  kernel of $r_A \colon \IC[\IZ^d]_{(0)}^r \to \IC[\IZ^d]_{(0)}^s$. We can choose a matrix
  $B \in M_{q,r}(\IC[\IZ^d]_{(0)})$ such that the sequence of $\IC[\IZ^d]_{(0)}$-modules
  $0 \to \IC[\IZ^d]_{(0)}^q \xrightarrow{r_B} \IC[\IZ^d]_{(0)}^r \xrightarrow{r_A}
  \IC[\IZ^d]_{(0)}^s$ is exact.  By possibly multiplying each entry of $B$ with the same
  non-trivial element in $\IC[\IZ^d]$, we can additionally arrange that $B$ lies
  $M_{q,r}(\IC[\IZ^d])$. We conclude from~\cite[Lemma~1.34~(1) on page~35]{Lueck(2002)}
  that the following sequence is weakly exact
\[
0 \to L^2(\IZ^d)^q \xrightarrow{r_B^{(2)}} L^2(\IZ^d)^r \xrightarrow{r_A^{(2)}} L^2(\IZ^d)^s.
\]
We consider it as a chain complex of Hilbert $\caln(\IZ^d)$-modules concentrated in
dimensions $2$, $1$ and $0$.  Then its first and second $L^2$-homology vanishes.  The
Laplace operators of it in dimensions $1$ and $2$ is given by $r_{D_1}^{(2)}$ and
$r_{D_2}^{(2)}$.  We conclude from~\cite[Lemma~3.39 on page~145]{Lueck(2002)} applied in the case where
we replace $L^2(\IZ^d)^s$ by the closure of the image of $r_A^{(2)}$ that
$r_{D_1}^{(2)}$ and $r_{D_2}^{(2)}$ are weak isomorphisms and in particular injective.
\\[2mm]~\eqref{lem:reducting_to_weak_autos:formula} We conclude from
Theorem~\ref{the:main_properties_of_det}~\eqref{the:main_properties_of_det:det(f)_is_det(fast)}
and from~\cite[Lemma~3.30 on page~140]{Lueck(2002)} applied to the chain complex of
Hilbert modules above
\begin{eqnarray*}
\lefteqn{\ln\bigl({\det}_{\caln(\IZ^d)}(r_A^{(2)})\bigr)}
& & 
\\
& = &
\ln\bigl({\det}_{\caln(\IZ^d)}(r_B^{(2)})\bigr)  - \frac{1}{2} \cdot \left(2 \cdot \ln\bigl({\det}_{\caln(\IZ^d)}(r_{D_2}^{(2)})\bigr) 
-  \ln\bigl({\det}_{\caln(\IZ^d)}(r_{D_1}^{(2)})\bigr)\right)
\\
& = &
\frac{1}{2} \cdot  \ln\bigl({\det}_{\caln(\IZ^d)}(r_{D_2}^{(2)})\bigr)  - \ln\bigl({\det}_{\caln(\IZ^d)}(r_{D_2}^{(2)})\bigr) 
+ \frac{1}{2} \cdot \ln\bigl({\det}_{\caln(\IZ^d)}(r_{D_1}^{(2)})\bigr)
\\
& = &
- \frac{1}{2} \cdot  \ln\bigl({\det}_{\caln(\IZ^d)}(r_{D_2}^{(2)})\bigr)  
+ \frac{1}{2} \cdot \ln\bigl({\det}_{\caln(\IZ^d)}(r_{D_1}^{(2)})\bigr)
\\
& = &
\ln\left(\sqrt{\frac{{\det}_{\caln(\IZ^d)}(r_{D_1}^{(2)})}{{\det}_{\caln(\IZ^d)}(r_{D_2}^{(2)} )}}\right).
\end{eqnarray*}
\eqref{lem:reducting_to_weak_autos:reducing_with_det_over_C[Zd]}
This follows from assertion~\eqref{lem:reducting_to_weak_autos:weak}
and Lemma~\ref{lem:det(C[G]} and
Remark~\ref{rem:det_abelian_case}.
This finishes the proof of Lemma~\ref{lem:reducting_to_weak_autos}.
\end{proof}

\begin{theorem}[Approximating Mahler measures for matrices over {$\IC[\IZ^d]$} by matrices over {$\IC[\IZ]$}]
\label{the:Approximating_Mahler_measures_for_matrices_over_IC[IZd]_by_matrices_over_IC[IZ]}
Consider a matrix $A \in M_{r,s}(\IC[\IZ^d])$. Then we get
\begin{multline*}
   {\det}_{\caln(\IZ^d)}\bigl(r_A^{(2)} \colon L^2(\IZ^d)^r \to L^2(\IZ^d)^s\bigr) 
    \\ =
    \lim_{k_2 \to \infty} \; \lim_{k_3 \to \infty} \; \ldots \; \lim_{k_d \to \infty} \;
    {\det}_{\caln(\IZ)}\bigl(r_{A[k_2, \ldots, k_d]}^{(2)} \colon L^2(\IZ)^r \to
    L^2(\IZ)^s\bigr).
  \end{multline*}
\end{theorem}
\begin{proof}
  Because of Lemma~\ref{lem:reducting_to_weak_autos}, we can choose for $A$ an integer $q \ge 0$ and a matrix 
  $B \in M_{q,r}(\IC[\IZ^d])$ such that  the
  sequence $0 \to L^2(\IZ)^q \xrightarrow{r_B^{(2)}} L^2(\IZ^d)^r \xrightarrow{r_A^{(2)}}
  L^2(\IZ^d)^s$ is weakly exact. Put $D_1 = B^*B + AA^*$ in $M_{r,r}(\IC[\IZ^d])$ and 
 $D_2  = BB^*$ in $M_{q,q}(\IC[\IZ^d])$ as in in Lemma~\ref{lem:reducting_to_weak_autos}. We
  conclude from Lemma~\ref{lem:reducting_to_weak_autos} that ${\det}_{\IC[\IZ^d]}(D_1)$
  and ${\det}_{\IC[\IZ^d]}(D_2)$ in $\IC[\IZ^d]$ are different from zero and we get
\begin{eqnarray}
{\det}_{\caln(\IZ^d)}(r_{D_1}^{(2)}) 
& = & 
{\det}_{\caln(\IZ^d)}\bigl(r_{{\det}_{\IC[\IZ^d]}(D_1)}^{(2)} \colon L^2(\IZ^d) \to L^2(\IZ^d)\bigr);
\label{the:Approximating_Mahler_measures_for_matrices_over_IC[IZd]_by_matrices_over_IC[IZ]:(1)}
\\
{\det}_{\caln(\IZ^d)}(r_{D_2}^{(2)}) 
& = & 
{\det}_{\caln(\IZ^d)}\bigl(r_{{\det}_{\IC[\IZ^d]}(D_2)}^{(2)} \colon L^2(\IZ^d) \to L^2(\IZ^d)\bigr);
\label{the:Approximating_Mahler_measures_for_matrices_over_IC[IZd]_by_matrices_over_IC[IZ]:(2)}
\\
{\det}_{\caln(\IZ^d)}(r_{A}^{(2)}) 
& = & 
\sqrt{\frac{{\det}_{\caln(\IZ^d)}\bigl(r_{{\det}_{\IC[\IZ^d]}(D_1)}^{(2)} \colon L^2(\IZ^d) \to L^2(\IZ^d)\bigr)}
{{\det}_{\caln(\IZ)}\bigl(r_{{\det}_{\IC[\IZ^d]}(D_2)}^{(2)} \colon L^2(\IZ^d) \to L^2(\IZ^d)\bigr)}}.
\label{the:Approximating_Mahler_measures_for_matrices_over_IC[IZd]_by_matrices_over_IC[IZ]:(3)}
\end{eqnarray}

Fix $i = 1, \ldots, d-1$. Let $b_i$ be the maximum over the norms of those integers $n_i$ for which 
there exists a monomial of the form $z_1^{n_1}z_2^{n_2} \cdots z_d^{n_d}$ 
such that its coefficients in the description of ${\det}_{\IC[\IZ^d]}(D_1)$ and ${\det}_{\IC[\IZ^d]}(D_2)$
as a sum of such monomials is non-trivial. Define $c_i = 2 \cdot \sum_{j = 1}^ib_j$ for $i = 1,2 \ldots, (d-1)$. 

Consider any sequence of natural numbers $k_1, k_2, \ldots, k_d$ satisfying 
$k_2 >  c_1$, $k_3 >  c_2 \cdot k_2$, $k_4 > c_3 \cdot k_3$, $\ldots$,  $k_d > c_{d-1} \cdot k_{d-1}$.  
Let $l = 1,2$. Then we get for any two $d$-tuples $(m_1, m_2, \ldots m_d)$ and
$(n_1, n_2, \ldots , n_d)$ for which there exists a monomial of the form 
$z_1^{m_1}z_2^{m_2} \cdots z_d^{m_d}$  and $z_1^{n_1}z_2^{n_2} \cdots z_d^{n_d}$ 
such that its coefficients in the description of ${\det}_{\IC[\IZ^d]}(D_l)$
as a sum of such monomials is non-trivial
\begin{multline*}
q(k_1, k_3, \ldots , k_d)\bigl((m_1, m_2, \ldots m_d)\big) 
= 
q(k_1, k_3, \ldots , k_d)\bigl((n_1, n_2, \ldots n_d)\big) 
\\
\implies (m_1, m_2, \ldots m_d) = (n_1, n_2, \ldots n_d).
\end{multline*}
This implies that  $\bigl({\det}_{\IC[\IZ^d]}(D_l)\bigr)[k_2, \ldots, k_d]$ is  not zero.  One easily checks
\begin{eqnarray*}
{\det}_{\IC[\IZ^d]}(D_1)[k_2, \ldots, k_d]
& = & 
{\det}_{\IC[\IZ^d]}\bigl(D_1[k_2, \ldots, k_d]\bigr);
\\
{\det}_{\IC[\IZ^d]}(D_2)[k_2, \ldots, k_d]
& = & 
{\det}_{\IC[\IZ^d]}\bigl(D_2[k_2, \ldots, k_d]\bigr);
\\
D_1[k_2, \ldots, k_d] 
& = & 
B[k_2, \ldots, k_d]^* B[k_2, \ldots, k_d] + A[k_2, \ldots, k_d]A[k_2, \ldots, k_d]^*;
\\
D_2[k_2, \ldots, k_d] 
& = & 
B[k_2, \ldots, k_d]B[k_2, \ldots, k_d]^*.
\end{eqnarray*}
We conclude that ${\det}_{\IC[\IZ]}\bigl(D_1[k_2, \ldots, k_d]\bigr)$ and ${\det}_{\IC[\IZ]}\bigl(D_2[k_2, \ldots, k_d]\bigr)$
are non-trivial. Lemma~\ref{lem:det(C[G]}~\eqref{lem:det(C[G]:equivalent} implies 
that $r_{D_1[k_2, \ldots, k_d]}^{(2)} \colon L^2(\IZ)^r \to L^2(\IZ)^r$ and $r_{D_2[k_2, \ldots, k_d]}^{(2)} \colon L^2(\IZ)^q \to L^2(\IZ)^q$
are weak isomorphism. This implies that the sequence
$0 \to L^2(\IZ)^q \xrightarrow{r_{B[k_2, \ldots, k_d]}^{(2)}} L^2(\IZ)^r \xrightarrow{r_{A[k_2, \ldots, k_d]}^{(2)}} L^2(\IZ)^s$
is weakly exact since we can view it as a chain complex of Hilbert $\caln(\IZ)$-modules,
its Laplace operator in dimension $1$ and $2$ is given by
the weak isomorphisms $r_{D_1[k_2, \ldots, k_d]}^{(2)}$ and $r_{D_2[k_2, \ldots, k_d]}^{(2)}$ and we can 
apply~\cite[Lemma~3.39 on page~145]{Lueck(2002)} applied in the case where
we replace $L^2(\IZ)^s$ by the closure of the image of $r_{A[k_2, \ldots, k_d]}^{(2)}$. We conclude from
Lemma~\ref{lem:reducting_to_weak_autos} applied to $A[k_2, \ldots, k_d]$ and $B[k_2, \ldots, k_d]$
\begin{multline*}
{\det}_{\caln(\IZ)}\bigl(r_{A[k_2, \ldots, k_d]}^{(2)}\colon L^2(\IZ)^r \to L^2(\IZ)^s\bigr) 
\\ = 
\sqrt{\frac{{\det}_{\caln(\IZ)}\bigl(r_{{\det}_{\IC[\IZ]}(D_1[k_2, \ldots, k_d])}^{(2)} \colon L^2(\IZ) \to L^2(\IZ)\bigr)}
{{\det}_{\caln(\IZ)}\bigl(r_{{\det}_{\IC[\IZ]}(D_2[k_2, \ldots, k_d])}^{(2)} \colon L^2(\IZ) \to L^2(\IZ)\bigr)}},
\end{multline*}
provided that $k_2 >  c_1$, $k_3 >  c_2 \cdot k_2$, $k_4 > c_3 \cdot k_3$, $\ldots$,  $k_d > c_{d-1} \cdot k_{d-1}$.  
This implies
\begin{multline}
\lim_{k_2 \to \infty} \; \lim_{k_3 \to \infty} \; \ldots \; \lim_{k_d \to \infty} \; 
{\det}_{\caln(\IZ)}\bigl(r_{A[k_2, \ldots, k_d]}^{(2)}\colon L^2(\IZ)^r \to L^2(\IZ)^s\bigr) 
\\
= 
\lim_{k_2 \to \infty} \; \lim_{k_3 \to \infty} \; \ldots \; \lim_{k_d \to \infty} \; 
\sqrt{\frac{{\det}_{\caln(\IZ)}\bigl(r_{{\det}_{\IC[\IZ]}(D_1[k_2, \ldots, k_d])}^{(2)} \colon L^2(\IZ) \to L^2(\IZ)\bigr)}
{{\det}_{\caln(\IZ)}\bigl(r_{{\det}_{\IC[\IZ]}(D_2[k_2, \ldots, k_d])}^{(2)} \colon L^2(\IZ) \to L^2(\IZ)\bigr)}}.
\label{the:Approximating_Mahler_measures_for_matrices_over_IC[IZd]_by_matrices_over_IC[IZ]:(4)}
\end{multline}

We get from~\eqref{approximation_rewritten_for_(1,1)-matrices},~%
\eqref{the:Approximating_Mahler_measures_for_matrices_over_IC[IZd]_by_matrices_over_IC[IZ]:(1)}
and~\eqref{the:Approximating_Mahler_measures_for_matrices_over_IC[IZd]_by_matrices_over_IC[IZ]:(2)}
\begin{multline*}
   {\det}_{\caln(\IZ^d)}\bigl(r_{{\det}_{\IC[\IZ^d]}(D_1)}^{(2)} \colon L^2(\IZ^d) \to L^2(\IZ^d)\bigr)
    \\ =
    \lim_{k_2 \to \infty} \; \lim_{k_3 \to \infty} \; \ldots \; \lim_{k_d \to \infty} \;
    {\det}_{\caln(\IZ)}\bigl(r_{{\det}_{\IC[\IZ]}(D_1[k_2, \ldots, k_d])}^{(2)} \colon L^2(\IZ) \to L^2(\IZ)\bigr),
      \end{multline*}
and
\begin{multline*}
   {\det}_{\caln(\IZ^d)}\bigl(r_{{\det}_{\IC[\IZ^d]}(D_2)}^{(2)} \colon L^2(\IZ^d) \to L^2(\IZ^d)\bigr)
    \\ =
    \lim_{k_2 \to \infty} \; \lim_{k_3 \to \infty} \; \ldots \; \lim_{k_d \to \infty} \;
    {\det}_{\caln(\IZ)}\bigl(r_{{\det}_{\IC[\IZ]}(D_2[k_2, \ldots, k_d])}^{(2)} \colon L^2(\IZ) \to L^2(\IZ)\bigr).
      \end{multline*}
This implies
\begin{multline}
   \sqrt{\frac{{\det}_{\caln(\IZ^d)}\bigl(r_{{\det}_{\IC[\IZ]}(D_1)}^{(2)} \colon L^2(\IZ^d) \to L^2(\IZ^d)\bigr)}
{{\det}_{\caln(\IZ^d)}\bigl(r_{{\det}_{\IC[\IZ]}(D_2)}^{(2)} \colon L^2(\IZ^d) \to L^2(\IZ^d)\bigr)}}
    \\ =
    \lim_{k_2 \to \infty} \; \lim_{k_3 \to \infty} \; \ldots \; \lim_{k_d \to \infty} \;
    \sqrt{\frac{{\det}_{\caln(\IZ)}\bigl(r_{{\det}_{\IC[\IZ]}(D_1[k_2, \ldots, k_d])}^{(2)} \colon L^2(\IZ) \to L^2(\IZ)\bigr)}
{{\det}_{\caln(\IZ)}\bigl(r_{{\det}_{\IC[\IZ]}(D_2[k_2, \ldots, k_d])}^{(2)} \colon L^2(\IZ) \to L^2(\IZ)\bigr)}}.
    \label{the:Approximating_Mahler_measures_for_matrices_over_IC[IZd]_by_matrices_over_IC[IZ]:(5)}
  \end{multline}
Putting~\eqref{the:Approximating_Mahler_measures_for_matrices_over_IC[IZd]_by_matrices_over_IC[IZ]:(3)},
\eqref{the:Approximating_Mahler_measures_for_matrices_over_IC[IZd]_by_matrices_over_IC[IZ]:(4)}
and~\eqref{the:Approximating_Mahler_measures_for_matrices_over_IC[IZd]_by_matrices_over_IC[IZ]:(5)}
together yields the desired equality
\begin{multline*}
   {\det}_{\caln(\IZ^d)}\bigl(r_A^{(2)} \colon L^2(\IZ^d)^r \to L^2(\IZ^d)^s\bigr) 
    \\ =
    \lim_{k_2 \to \infty} \; \lim_{k_3 \to \infty} \; \ldots \; \lim_{k_d \to \infty} \;
    {\det}_{\caln(\IZ)}\bigl(r_{A[k_2, \ldots, k_d]}^{(2)} \colon L^2(\IZ)^r \to
    L^2(\IZ)^s\bigr).
  \end{multline*}
This finishes the proof of Theorem~\ref{the:Approximating_Mahler_measures_for_matrices_over_IC[IZd]_by_matrices_over_IC[IZ]}.
\end{proof}


 \typeout{-------------   Section 9:  Lehmer's problem for groups with torsion --------------------}

\section{Lehmer's problem for groups with torsion}
\label{sec:Lehmer's_Question_for_groups_with_torsion}

\begin{remark}[Bound on the order of finite subgroups]\label{rem:bound_on_the_order_of_finite_subgroups}
  Let $G$ be a group with  $\Lambda^w_1 (G) > 1$. Let $H \subseteq G$ be
  any finite subgroup.  We conclude from
  Lemma~\ref{lem:elementary_properties_of_Lambda(G)}~\eqref{lem:elementary_properties_of_Lambda(G):subgroups}
  and~\eqref{lem:elementary_properties_of_Lambda(G):G_finite} that $\Lambda^w_1(G) \le
  \Lambda^w_1(H) \le (|H|-1)^{|H|^{-1}}$ holds if $|H| \ge 3$. Since we have 
   $\lim_{m \to \infty}  (m-1)^{m^{-1}} = 1$,  there is a constant $C > 1$ depending only on
  $\Lambda^w_1(G)$ such that $|H| \le C$ holds.
\end{remark}

\begin{example}[Lamplighter group]\label{exa:lamplighther_group}
  Let $L = \IZ/2 \wr \IZ$ be the lamplighter group. It is finitely generated and
  contains finite subgroups of arbitrary large order.  Hence we get $\Lambda(L) =
  \Lambda^w(L) = \Lambda_1(L) = \Lambda^w_1(L) = 1$ from Remark~\ref{rem:bound_on_the_order_of_finite_subgroups}.
  The lamplighter group is a subgroup of a finitely presented group, 
see for instance~\cite[Remark~10.24 on page~380]{Lueck(2002)}.
 Hence there exists a finitely presented group  $G$ with 
$\Lambda(G) =   \Lambda^w(G) = \Lambda_1(G) = \Lambda^w_1(G) = 1$ by 
Lemma~\ref{lem:elementary_properties_of_Lambda(G)}~%
\eqref{lem:elementary_properties_of_Lambda(G):subgroups}.
\end{example}

The considerations above explain why we will concentrate on torsionfree groups in the sequel.


 \typeout{-----------------------   Section 10: Finitely generated free abelian groups --------------------}

\section{Finitely generated free abelian groups}
\label{sec:Finitely_generated_free_abelian_groups}

\begin{theorem}[Finitely generated free abelian groups]
\label{the:Finitely_generated_free_abelian_groups}
Let $d \ge 1$ be an integer. Then we have:
\[
\Lambda_1(\IZ^d)   =  \Lambda_1^w (\IZ^d) = \Lambda^w(\IZ^d)   
= \Lambda_1(\IZ) =\Lambda_1^w(\IZ)  =   \Lambda^w(\IZ),
\]
and
\[
\Lambda(\IZ^d) = \Lambda(\IZ).
\]

\end{theorem}
\begin{proof}
Theorem~\ref{the:Finitely_generated_free_abelian_groups} follows after we have shown 
the following equalities
\begin{eqnarray}
\Lambda(\IZ^d)  & =&  \Lambda(\IZ);
\label{the:Finitely_generated_free_abelian_groups:from_Zd_to_Z}
\\
\Lambda^w_1(\IZ^d)  & = & \Lambda_1(\IZ^d);
\label{the:Finitely_generated_free_abelian_groups:w,1_is_1}
\\
\Lambda^w(\IZ^d)  & =&  \Lambda_1^w(\IZ^d); 
\label{the:Finitely_generated_free_abelian_groups:w_is_w,1}
\\
\Lambda_1^w(\IZ^d)  & = & \Lambda_1^w(\IZ).
\label{the:Finitely_generated_free_abelian_groups:_from_d_to_1}
\end{eqnarray}

Equality~\eqref{the:Finitely_generated_free_abelian_groups:from_Zd_to_Z}
follows from Theorem~\ref{the:Approximating_Mahler_measures_for_matrices_over_IC[IZd]_by_matrices_over_IC[IZ]}.

Equation~\eqref{the:Finitely_generated_free_abelian_groups:w,1_is_1} 
follows from the conclusion of 
Lemma~\ref{lem:det(C[G]}~\eqref{lem:det(C[G]:equivalent}
that for $a \in \IZ[\IZ^d]$ with $a \not= 0$ the map $r_a^{(2)} \colon L^2(\IZ^d) \to L^2(\IZ^d)$ is injective. 

Equation~\eqref{the:Finitely_generated_free_abelian_groups:w_is_w,1} has already been proved in
Lemma~\ref{lem:lambda_upper_w(Z_upper_d)_is_lambda_upper_w_1(Z_upper_d)}.

Equation~\eqref{the:Finitely_generated_free_abelian_groups:_from_d_to_1} follows from
Lemma~\ref{lem:elementary_properties_of_Lambda(G)}~\eqref{lem:elementary_properties_of_Lambda(G):subgroups} and
Remark~\ref{rem:Approximating_Mahler_measures_for_polynomials_in_several_variables},
see in particular~\eqref{approximation_rewritten_for_(1,1)-matrices}.
\end{proof}


\typeout{----------------------------   Section 11:  Non-abelian free groups ----------------------------------}

\section{Non abelian free groups}
\label{sec:Non_abelian_Free_groups}
We cannot prove anything for non-abelian free groups so far, but we want at least to list some good
properties and explain the main problem. Throughout this section $F$ denotes a non-abelian free group.

A ring $R$ is called \emph{semifir} if every finitely generated submodule of a free
$R$-module is free and for a free module any two basis have the same cardinality. For any field $K$ 
the group ring  $K[F]$ is a semifir~\cite[Corollary on page 68]{Cohn(1964)},
\cite{Dicks-Menal(1979)}. Moreover, $\IC G$ embeds into a skewfield $\cald(F)$ and the
inclusion is $\Sigma(\IC G \subseteq \calu(G))$-inverting as explained
in~\cite[Section~10.3.6 and Lemma~10.82 on page~408]{Lueck(2002)}.

A problem is that for a matrix $A \in M_{r,s}(\IC F)$ it can happen that $r_A
\colon \IC G^r \to \IC G^s$ is injective but $r_A^{(2)} \colon L^2(G)^r \to L^2(G)^s$ is
not injective. An example comes for the free group $F_r$ in $r \ge 2$ generators $s_1, s_2, \ldots , s_r$
from the $\IC[F_r]$-chain complex of the universal covering
of $\bigvee_{i=1}^r S^1$ whose first differential $r_A \colon \IC [F_r]^r \to  \IC [F_r]$ is given by the
matrix $A$ which is the transpose of  $(s_1 -1, s_2 -1, \ldots , s_r-1)$. The map $r_A$ is injective since the universal
covering is $1$-connected. The induced $F$-equivariant bounded operator 
$r_A^{(2)} \colon L^2(F_r)^r \to L^2(F_r)$ is not injective since Lemma~\ref{lem:weak_exactness} implies
\[
\dim_{\caln(F_r)}\bigl(\ker(r_A^{(2)})\bigr) \ge  \dim_{\caln(F_r)}\bigl(L^2(F_r)^r\bigr) - \dim_{\caln(F)}\bigl(L^2(F_r)\bigr) = r - 1 > 0.
\]
Actually we have $\dim_{\caln(F_r)}\bigl(\ker(r_A^{(2)})\bigr)  = r - 1$.


 \typeout{-------------   Section 12:  Lehmer's problem for torsionfree groups  --------------------}

\section{Lehmer's problem for torsionfree groups}
\label{sec:Lehmer's_Question_for_torsionfree_groups}

Theorem~\ref{the:Finitely_generated_free_abelian_groups} leads to state the following
version of Lehmer's problem for torsionfree groups.

\begin{problem}[Lehmer's problem for torsionfree groups.]
For which torsionfree group $G$  does
\[
\Lambda^w(G) > 1
\]
hold?
\end{problem}


 \typeout{---------------------------------   Section 13:  3-manifolds-----------------------------}

\section{$3$-manifolds}
\label{sec:3-manifolds}

Let $M$ be a closed hyperbolic $3$-manifold with fundamental group $\pi$.  Then the
$L^2$-torsion of its universal covering satisfies
\[
- \rho^{2}(\widetilde{M}) = \frac{1}{6\pi} \cdot \vol(M)
\]
where $\vol(M)$ is the volume, see for instance~\cite[Theorem~3.152 on page~187]{Lueck(2002)}. 
There is a natural number $n$ and a $(n,n)$-matrix $A$ with
entries in $\IZ \pi$ such that $r_A^{(2)} \colon L^2(\pi)^n \to L^2(\pi)^n$ is a weak
isomorphisms and its Fulgede-Kadison determinant satisfies
\[
- \rho^{2}(\widetilde{M}) = \ln (\det^{(2)}(r_A^{(2)}).
\]
This follows from the argument in the proof of~\cite[Theorem~2.4]{Lueck(1994a)}. We conclude
\[
{\det}_{\caln(\pi)}^{(2)}(r_A^{(2)} = \exp\left(\frac{1}{6\pi} \cdot \vol(M)\right).
\]
Hence we get 
\begin{equation}
\Lambda^w(\pi) \le \exp\left(\frac{1}{6\pi} \cdot \vol(M)\right).
\label{estimate_for_Lambda(pi)_for_pi_is_pi(closed_hyperbolic_3-manifold}
\end{equation}

\begin{example}[Weeks manifold]\label{exa:Week's_manifold}
  There is a closed hyperbolic $3$-manifold $W$, the so called Weeks manifold, which is
  the unique closed hyperbolic $3$-manifold with smallest volume,
  see~\cite[Corollary~1.3]{Gabai_Meyerhoff_Milley(2009)}.  Its volume is between 0,942 and
  0,943. Hence we get
  from~\eqref{estimate_for_Lambda(pi)_for_pi_is_pi(closed_hyperbolic_3-manifold} the
  estimate
  \[
    \Lambda^w(\pi) \le \exp\left(\frac{1}{6\pi} \cdot 0,943 \right) \le 1,06
  \]
  This implies $\Lambda^w(\pi) < M(L)$.
\end{example}


 \typeout{---------------   Section 14:  Appendix:  Fuglede-Kadison determinants --------------------}

\section{Appendix: $L^2$-invariants}
\label{sec:L2-invariants}

In this appendix we give some basic definitions and properties about $L^2$-invariants.


\subsection{Group von Neumann algebras}
\label{subsec:Group_von_Neumann_algebras}

Denote by $L^2(G)$ the Hilbert space $L^2(G)$ consisting of formal sums 
$\sum_{g \in G} \lambda_g \cdot g$ for complex numbers $\lambda_g$ 
such that $\sum_{g \in G} |\lambda_g|^2 < \infty$.  
This is the same as the Hilbert space completion of the complex group ring
$\IC G$ with respect to the pre-Hilbert space structure for which $G$ is an orthonormal
basis.  Notice that left multiplication with elements in $G$ induces an isometric
$G$-action on $L^2(G)$. Given a Hilbert space $H$, denote by $\calb(H)$ the $C^*$-algebra
of bounded (linear) operators from $H$ to itself, where the norm is the operator norm and
the involution is given by taking adjoints.

\begin{definition}[Group von Neumann algebra]
\label{def:group_von_Neumann_algebra}
The \emph{group von Neumann algebra}
$\caln(G)$ of the group $G$ is defined
as the algebra of $G$-equivariant bounded operators
from $L^2(G)$ to $L^2(G)$
\begin{eqnarray*}
\caln(G)
& := &
\calb(L^2(G))^G.
\end{eqnarray*}
\end{definition}

In the sequel we will view the complex group ring $\IC G$ as a subring of $\caln(G)$ by
the embedding of $\IC$-algebras $\rho_r \colon \IC G \to \caln(G)$ which sends $g \in G$
to the $G$-equivariant operator $r_{g^{-1}} \colon L^2(G) \to L^2(G)$ given by right
multiplication with $g^{-1}$.

\begin{remark}[The general definition of von Neumann algebras]\label{rem:definition_of_a_von_Neumann_algebra}
  In general a \emph{von Neumann algebra}   $\cala$
  is a sub-$\ast$-algebra of $\calb(H)$ for some Hilbert space $H$, which is closed in the
  weak topology and contains $\id\colon H \to H$. Often in the literature the group von
  Neumann algebra $\caln(G)$ is defined as the closure in the weak topology of the complex
  group ring $\IC G$ considered as $\ast$-subalgebra of $\calb(L^2(G))$. This definition
  and Definition~\ref{def:group_von_Neumann_algebra} agree,
  see~\cite[Theorem~7.2 on page~434]{Kadison-Ringrose(1986)}.
\end{remark}

\begin{example}[The von Neumann algebra of a finite group]\label{exa:group_von_Neumann_algebra_of_a_finite_group}
  If $G$ is finite, then
  nothing happens, namely $\IC G = L^2(G) = \caln(G)$. 
\end{example}

\begin{example}[The von Neumann algebra of $\IZ^d$]\label{exa:group_von_neumann_algebra_of_Zd}
  In general there is no concrete model for $\caln(G)$.  However, for $G = \IZ^d$, there
  is the following illuminating model for the group von Neumann algebra $\caln(\IZ^d)$.
  Let $L^2(T^d)$  be the Hilbert space of equivalence classes of $L^2$-integrable complex-valued functions
  on the $n$-dimensional torus $T^d$, where two such functions are called equivalent if
  they differ only on a subset of measure zero.  Define the ring $L^{\infty}(T^d)$ by
  equivalence classes of essentially bounded measurable functions $f\colon T^d \to \IC$,
  where essentially bounded  means that there is a constant $C > 0$ such that the set 
  $\{x \in T^d\mid |f(x)| \ge  C\}$ has measure zero. An element $(k_1, \ldots , k_n)$ in $\IZ^d$ acts 
  isometrically on $L^2(T^d)$ by pointwise multiplication with the function $T^d \to \IC$, which maps
  $(z_1, z_2, \ldots, z_n)$ to $z_1^{k_1} \cdot \cdots \cdot z_n^{k_n}$.  The Fourier
  transform yields an isometric $\IZ^d$-equivariant isomorphism $L^2(\IZ^d)
  \xrightarrow{\cong} L^2(T^d)$.  Hence $\caln(\IZ^d) = \calb(L^2(T^d))^{\IZ^d}$.  We
  obtain an isomorphism (of $C^*$-algebras)
  \[
  L^{\infty}(T^d) \xrightarrow{\cong} \caln(\IZ^d)
  \]
  by sending $f \in L^{\infty}(T^d)$ to the $\IZ^d$-equivariant operator $M_f\colon
  L^2(T^d) \to L^2(T^d), \; g \mapsto g \cdot f,$ where $(g\cdot f)(x)$ is defined by
  $g(x)\cdot f(x)$.
\end{example}


\subsection{The von Neumann dimension}
\label{subsec:The_von_Neumann_dimension}

An important feature of the group von Neumann algebra is its trace.

\begin{definition}[Von Neumann trace]\label{def:trace_of_the_group_von_Neumann_algebra}
  The \emph{von Neumann trace} on $\caln(G)$ is defined by
\[
\tr_{\caln(G)}\colon \caln(G) \to \IC, \quad f \mapsto \langle f(e),e\rangle_{L^2(G)},
\]
where $e \in G \subseteq L^2(G)$ is the unit element.
\end{definition}

\begin{definition}[Finitely generated Hilbert module]\label{def:finitely_generated_Hilbert_module}
A finitely generated \emph{Hilbert $\caln(G)$-module}
$V$ is a Hilbert space $V$ together
with a linear isometric $G$-action  such that there exists
an isometric linear $G$-embedding
of $V$ into $L^2(G)^r$ for some natural number $r$.
A \emph{map of Hilbert $\caln(G)$-modules}
$f\colon V \to W$ is a bounded $G$-equivariant operator.
\end{definition}

\begin{definition}[Von Neumann dimension]\label{def:von_Neumann_dimension}
  Let $V $ be a finitely generated Hilbert $\caln(G)$-module.  Choose a matrix 
  $A =  (a_{i,j}) \in M_{r,r}(\caln(G))$ with $A^2 = A$ such that the image of the $G$-equivariant
  bounded operator $r_A^{(2)} \colon L^2(G)^r \to L^2(G)^r$ given by $A$ is isometrically
  $G$-isomorphic to $V$.  Define the \emph{von Neumann dimension} of $V$ by
  \begin{eqnarray*}
    \dim_{\caln(G)}(V)
    & := & \sum_{i=1}^r \tr_{\caln(G)}(a_{i,i})
    \quad \in [0,\infty).
  \end{eqnarray*}
\end{definition}

The von Neumann dimension $\dim_{\caln(G)}(V)$ depends only on the isomorphism class of
the Hilbert $\caln(G)$-module $V$ but not on the choice of $r$ and the matrix $A$.
The von Neumann dimension $\dim_{\caln(G)}$ is \emph{faithful}, i.e. 
$\dim_{\caln(G)}(V) = 0 \Leftrightarrow V= 0$ holds for any finitely generated Hilbert 
$\caln(G)$-module $V$. It is weakly exact in the following sense, 
see~\cite[Theorem~1.12 on page~21]{Lueck(2002)}.

\begin{lemma}\label{lem:weak_exactness}
  Let $0 \to V_0 \xrightarrow{i} V_1 \xrightarrow{p} V_2 \to 0$ be a sequence of finitely
  generated Hilbert $\caln(G)$-modules.  Suppose that it is weakly exact, i.e., $i$ is
  injective, the closure of $i$ is the kernel of $p$ and the image of $p$ is dense.  Then
  \[
  \dim_{\caln(G)}(V_1) = \dim_{\caln(G)}(V_0) + \dim_{\caln(G)}(V_0).
  \]
\end{lemma}

\begin{example}[Von Neumann dimension for finite groups]
\label{exa:Von_Neumann_dimension_for_finite_groups}
If $G$ is finite, then
$\dim_{\caln(G)}(V)$ is $\frac{1}{|G|}$-times the complex dimension
of the underlying complex vector space $V$.  
\end{example}

\begin{example}[Von Neumann dimension for $\IZ^d$] 
\label{exa:Von_Neumann_dimension_for_Zd}
Let $X \subset T^d$ be any measurable set and $\chi_X \in L^{\infty}(T^d)$
be its characteristic function. Denote by
$M_{\chi_X}\colon  L^2(T^d) \rightarrow  L^2(T^d)$
the $\IZ^d$-equivariant unitary projection given by
multiplication with $\chi_X$. Its image $V$ is a
Hilbert $\caln(\IZ^d)$-module with $\dim_{\caln(\IZ^d)}(V) = \vol(X)$.
\end{example}


\subsection{Weak isomorphisms}
\label{subsec:Weak_isomorphisms}

A bounded $G$-equivariant operator $f \colon L^2(G)^r \to L^2(G)^s$ is called a 
\emph{weak  isomorphism} if and only if it is injective and its image is dense.  If there exists a
weak isomorphism $L^2(G)^r \to L^2(G)^s$, then we must have $r = s$ by Lemma~\ref{lem:weak_exactness}. 
The following statements are equivalent for a bounded $G$-equivariant operator 
$f \colon L^2(G)^r \to L^2(G)^r$, see~\cite[Lemma~1.13 on page~23]{Lueck(2002)}:

\begin{enumerate}

\item $f$ is a weak isomorphism;

\item Its adjoint $f^*$ is a weak isomorphism;

\item $f$ is injective;

\item $f$ has dense image;

\item The von Neumann dimension of the closure of the image of $f$ is $r$.

\end{enumerate}

Consider a matrix $A \in M_{r,s}(\IC G)$. If $r_A^{(2)} \colon L^2(G)^r \to L^2(G)^s$ is
injective, then the $\IC G$-homomorphism $r_A \colon \IC G^r \to \IC G^s$ is
injective. The converse is not true in general but in the special case that $G$ is
amenable, this follows from~\cite[Theorem~6.24 on page~249 and Theorem~6.37 on page~259]{Lueck(2002)}.


\subsection{The Fuglede-Kadison determinant}
\label{subsec:The_Fuglede-Kadison_Determinant}

\begin{definition}[Spectral density function]
\label{def:spectral_density_function}
Let $f \colon V \to V$ be a morphisms of finitely generated Hilbert $\caln(G)$-modules.
Denote by $\{E_{\lambda}^{f^*f} \mid \lambda \in \IR\}$ 
the (right-continuous)
family of spectral projections of the positive operator $f^*f$.
Define the \emph{spectral density function of} $f$ by
\[
F_f \colon \IR \to [0,\infty) \quad \lambda \mapsto \dim_{\caln(G)}\bigl(\im(E_{\lambda^2}^{f^*f})\bigr).
\]
\end{definition}

The spectral density function is monotone non-decreasing and right-continuous.

\begin{example}[Spectral density function for finite groups]
\label{exa:spectral_density_function_for_finite_groups}
Let $G$ be finite and $f\colon U \to V$ be a map of finitely generated
Hilbert $\caln(G)$-modules, i.e., of finite-dimensional unitary $G$-representations.
Then $F(f)$ is the right-continuous step function
whose value at $\lambda$ is the sum of the complex dimensions of the eigenspaces
of $f^*f$ for eigenvalues $\mu \le \lambda^2$ divided by the order
of $G$, or, equivalently, the sum of the complex dimensions of the eigenspaces
of $|f|$ for eigenvalues $\mu \le \lambda$ divided by the order
of $G$.
\end{example}

\begin{example}[Spectral density function for $\IZ^d$]
\label{exa:spectral_density_function_for_Zd}
Let $G = \IZ^d$. In the sequel we use the notation and the identification
$\caln(\IZ^d) = L^{\infty}(T^d)$ of Example~\ref{exa:group_von_neumann_algebra_of_Zd}. 
For $f \in L^{\infty}(T^d)$
the spectral density function $F(M_f)$ of $M_f\colon L^2(T^d) \to L^2(T^d)$
sends $\lambda$ to the volume of the set
$\{z \in T^d \mid |f(z)| \le \lambda\}$.
\end{example}

\begin{definition}[Fuglede-Kadison determinant]\label{def:Fuglede-Kadison_determinant}
  Let $f\colon V \to V$ be a morphism of finitely generated Hilbert $\caln(G)$-modules.
   Let $F_f(\lambda)$ be the
  spectral density function of Definition~\ref{def:spectral_density_function} which is a
  monotone non-decreasing right-continuous function. Let $dF$ be the unique measure on the
  Borel $\sigma$-algebra on $\IR$ which satisfies $dF(]a,b]) = F(b)-F(a)$ for $a <  b$. 
  Then define the \emph{Fuglede-Kadison determinant} 
  \[
   {\det}_{\caln(G )}(f) \in [0,\infty)
   \]
to be the positive real number
\[
{\det}_{\caln(G )}(f) =  \exp\left(\int_{0+}^{\infty} \ln(\lambda) \; dF\right)
\]
if the Lebesgue integral $\int_{0+}^{\infty} \ln(\lambda) \; dF$ converges to a real number
and by $0$ otherwise.
\end{definition}

Notice that in the definition above we do not require that the source and domain of $f$
agree or that $f$ is injective or that $f$ is surjective. Our conventions imply that the Fulgede-Kadison
operator of the zero operator $0 \colon V \to V$ is $1$.

\begin{example}[Fuglede-Kadison determinant for finite groups] 
\label{exa:det_for_finite_groups} 
To illustrate this definition, we look at
the example where $G$ is finite. We essentially get
the classical determinant $\det_{\IC}$. Namely,
let $\lambda_1$, $\lambda_2$, $\ldots$, $\lambda_r$ be the non-zero
eigenvalues of $f^{\ast}f$ with multiplicity $\mu_i$. Then
one obtains, if $\overline{f^{\ast}f}$ is the automorphism
of the orthogonal complement of the kernel of $f^{\ast}f$ induced
by $f^{\ast}f$,
\begin{multline*}
{\det}_{\caln(G)}(f)  
= \exp\left(\sum_{i = 1}^r \frac{\mu_i}{|G|} \cdot \ln(\sqrt{\lambda_i})\right)
= \prod_{i=1}^r \lambda_i^{\frac{\mu_i}{2\cdot |G|}}
=  {\det}_{\IC}\bigl(\overline{f^{\ast}f}\big)^{\frac{1}{2\cdot |G|}}.
\end{multline*}
where ${\det}_{\IC}\bigl(\overline{f^{\ast}f})$ is put to be $1$ of $f$ is the zero operator
and hence $\overline{f^{\ast}f}$ is $\id \colon \{0\} \to \{0\}$.
If $f \colon \IC G^m \to \IC G^m$ is an automorphism, we get
\[
{\det}_{\caln(G)}(f)  = 
\left|{\det}_{\IC}(f)\right|^{\frac{1}{|G|}}. 
\]
\end{example}

\begin{example}[Fuglede-Kadison determinant for $(2,2)$-matrices over the trivial group] 
\label{exa:det_for_(2,2)-matrices_over_the-trivial_group}
Consider  $A \in M_{2,2}(\IC)$. Let $A^*$ be the conjugate transpose of $A$ and denote by $\tr_{\IZ}(AA^*)$ the trace of $AA^*$.
We conclude from Example~\ref{exa:det_for_finite_groups} 
for the trivial group $\{1\}$
\[{\det}_{\caln(\{1\})}\bigl(r_A \colon \IC^2 \to \IC^2\bigr)
=
\left\{\begin{array}{lcl}
|{\det}_{\IC}(A)| & \text{if} & \det_{\IC}(A) \not= 0;
\\
\sqrt{\tr_{\IZ}(AA^*)} & \text{if} & \det_{\IC}(A) = 0 \; \text{and}\; A \not= 0;
\\
1 & \text{if} & A = 0.
\end{array}\right.
\]
\end{example}

\begin{example}[Fuglede-Kadison determinant for $\caln(\IZ^d)$]
\label{exa:Fuglede-Kadison_determinant_for_G_is_Zd} 
Let $G = \IZ^d$. We use the  identification
$\caln(\IZ^d) = L^{\infty}(T^d)$ of Example~\ref{exa:group_von_neumann_algebra_of_Zd}.
For $f \in L^{\infty}(T^n)$ we conclude from
Example~\ref{exa:spectral_density_function_for_Zd}
\[
{\det}_{\caln(\IZ^d)}\left(M_f\colon  L^2(T^d) \to L^2(T^d)\right)
= 
\exp\left(\int_{T^d} \ln(|f(z)|) \cdot
\chi_{\{u \in S^1\mid f(u) \not= 0\}} \;dvol_z\right)
\]
using the convention  $\exp(-\infty) = 0$.
\end{example}

Let $i\colon H \rightarrow G$ be an injective group homomorphism.  Let $V$ be a finitely
generated Hilbert $\caln(H)$-module. There is an obvious pre-Hilbert structure on $\IC G
\otimes_{\IC H} V$ for which $G$ acts by isometries since $\IC G \otimes_{\IC H} V$ as a
complex vector space can be identified with $\bigoplus_{G/H} V$. Its Hilbert space
completion is a finitely generated Hilbert $\caln (G)$-module and denoted by $i_*M$. A morphism
of finitely generated Hilbert $\caln(H)$-modules $f\colon V \to W$ induces a map of finitely generated Hilbert
$\caln(G)$-modules $i_*f\colon i_*V \to i_*W$.

The following theorem can be found with proof in~\cite[Theorem~3.14 on~page~128
and Lemma~3.15~(4) on  page~129]{Lueck(2002)}.

\begin{theorem}[Fuglede-Kadison determinant]
\label{the:main_properties_of_det}\
\begin{enumerate}
\item
\label{the:main_properties_of_det:composition}
Let $f\colon  U \to V$ and $g\colon  V \to W$ be morphisms
of finitely generated Hilbert $\caln(G)$-modules such that $f$ has dense image and
$g$ is injective. Then
\[
{\det}_{\caln(G)}(g \circ f)  =  {\det}_{\caln(G)}(f) \cdot {\det}_{\caln(G)}(g);
\]

\item
\label{the:main_properties_of_det:additivity}
Let $f_1\colon  U_1 \to V_1$, $f_2\colon  U_2 \to V_2$ and
$f_3\colon  U_2 \to V_1$ be morphisms of finitely generated  Hilbert $\caln(G)$-modules
such that  $f_1$ has dense image and $f_2$ is injective. Then
\[
{\det}_{\caln(G)}\squarematrix{f_1}{f_3}{0}{f_2}
= 
{\det}_{\caln(G)}(f_1) \cdot {\det}_{\caln(G)}(f_2);
\]

\item\label{the:main_properties_of_det:det(f)_is_det(fast)}
Let $f\colon  U \to V$ be a morphism
of finitely generated Hilbert $\caln(G)$-modules. Then
\[
{\det}_{\caln(G)}(f) = {\det}_{\caln(G)}(f^*) = \sqrt{{\det}_{\caln(G)}(f^*f)} = \sqrt{{\det}_{\caln(G)}(ff^*)};
\]

\item
\label{the:main_properties_of_det:restriction}
Let $i \colon H \to  G$ be the inclusion of a subgroup of finite index $[G:H]$. Let 
$i^* f\colon i^* U \to i^* V$ be the morphism of 
finitely generated Hilbert $\caln(H)$-modules obtained from $f$
by restriction. Then
\begin{eqnarray*}
{\det}_{\caln(H)}(i^* f)
& = &
{\det}_{\caln(G)}(f)^{[G:H]};
\end{eqnarray*}

\item
\label{the:main_properties_of_det:induction}
Let $i\colon  H \to G$ be an injective group homomorphism and let
$f\colon  U \to V$ be a morphism of finitely generated Hilbert $\caln(H)$-modules.
Then 
\[
{\det}_{\caln(G)}(i_*f)  =  {\det}_{\caln(H)}(f).
\]

\end{enumerate}
\end{theorem}


\typeout{-------------------------------------- References  ---------------------------------------}



\end{document}